\documentclass{amsart}

\usepackage{hyperref}
\hypersetup{
    colorlinks=true,
    linkcolor=blue,
    citecolor=blue,
    urlcolor=blue,
}

\makeatletter
\newcommand\org@maketitle{}
\newcommand\@authors{}
\let\org@maketitle\maketitle
\def\maketitle{%
	\let\@authors\authors
	\nxandlist{; }{ and }{; }\@authors
	\hypersetup{
		linktocpage=true,
		pdftitle={\@title},
                pdfauthor={\@authors},
                pdfsubject={\subjclassname. \@subjclass},
		pdfkeywords={\@keywords}
	}%
	\org@maketitle
}
\makeatother

\usepackage{amssymb}
\usepackage{rsfso}
\usepackage{mathtools}
\usepackage{microtype}
\usepackage[alphabetic,msc-links]{amsrefs}
\usepackage{enumitem}
\usepackage{amsfonts}
\DeclareMathAlphabet{\mathcal}{OMS}{cmsy}{m}{n}

\usepackage{doi}
\renewcommand{\PrintDOI}[1]{\doi{#1}}

\numberwithin{equation}{section}

\newtheorem{theorem}{Theorem}[section]
\newtheorem{lemma}[theorem]{Lemma}
\newtheorem{corollary}[theorem]{Corollary}
\newtheorem{proposition}[theorem]{Proposition}
\theoremstyle{definition}
\newtheorem{definition}[theorem]{Definition}

\theoremstyle{remark}
\newtheorem{remark}[theorem]{Remark}

\newcommand{\cD}{{\mathcal D}}

\newcommand{\al}{\alpha}
\newcommand{\be}{\beta}

\newcommand{\la}{\lambda}
\newcommand{\ka}{\kappa}

\newcommand{\R}{\mathbb{R}}
\newcommand{\vp}{\varphi}

\newcommand{\bb}{\mathbf{b}}
\newcommand{\bg}{\mathbf{g}}
\newcommand{\bff}{\mathbf{f}}
\newcommand{\bq}{\mathbf{q}}

\newcommand{\D}{\nabla}

\renewcommand{\div}{\operatorname{div}}
\newcommand{\loc}{\mathrm{loc}}
\newcommand{\dist}{\operatorname{dist}}
\newcommand{\diam}{\operatorname{diam}}
\newcommand{\mean}[1]{\langle #1\rangle}
\newcommand{\lmean}[1]{\left\langle #1\right\rangle}

\def\Xint#1{\mathchoice
  {\XXint\displaystyle\textstyle{#1}}%
  {\XXint\textstyle\scriptstyle{#1}}%
  {\XXint\scriptstyle\scriptscriptstyle{#1}}%
  {\XXint\scriptscriptstyle\scriptscriptstyle{#1}}%
  \!\int}
\def\XXint#1#2#3{{\setbox0=\hbox{$#1{#2#3}{\int}$}
    \vcenter{\hbox{$#2#3$}}\kern-.5\wd0}}

\def\dashint{\Xint-}

\mathchardef\ordinarycolon\mathcode`\:
\mathcode`\:=\string"8000
\begingroup \catcode`\:=\active
  \gdef:{\mathrel{\mathop\ordinarycolon}}
\endgroup

\author{Hongjie Dong}
\author{Seongmin Jeon}
\author{Stefano Vita}
\address[H. Dong]{Division of Applied Mathematics
\newline\indent
Brown University
\newline\indent
182 George Street, Providence RI 02912, USA}
\email{hongjie\_dong@brown.edu}
\address[S. Jeon]{Department of Mathematics Education
\newline\indent
Hanyang University
\newline\indent
222 Wangsimni-ro, Seongdong-gu, Seoul 04763, Republic of Korea}
\email{seongminjeon@hanyang.ac.kr}
\address[S. Vita]{Dipartimento di Matematica ``F. Casorati"
\newline\indent
Universit\`a di Pavia
\newline\indent
Via Ferrata 5, 27100, Pavia, Italy}
\email{stefano.vita@unipv.it}

\title[Schauder type estimates for degenerate or singular equations]{Schauder type estimates for degenerate or singular elliptic equations with DMO coefficients}

\allowdisplaybreaks

\begin{document}
\begin{abstract}
    In this paper, we study degenerate or singular elliptic equations in divergence form
    $$
    -\div(x_n^\al A\D u)=\div(x_n^\al\bg)\quad\text{in }B_1\cap\{x_n>0\}.
    $$
    When $\al>-1$, we establish boundary Schauder type estimates under the conormal boundary condition on the flat boundary, provided that the coefficients satisfy Dini mean oscillation (DMO) type conditions. Additionally, as an application, we derive higher-order boundary Harnack principles for uniformly elliptic equations in divergence form  with DMO coefficients.
\end{abstract}

\maketitle
%\tableofcontents

\section{Introduction and main results}
\subsection{Degenerate or singular equations}
For a fixed number $\al\in(-1,\infty)$, we consider a second-order elliptic equation in divergence form with conormal boundary condition
\begin{equation}
    \label{eq:deg/sing-pde}
    \begin{cases}
    &-\div(x_n^\al A\D u)=\div(x_n^\al\bg)\\
    &\lim_{x_n\to0}x_n^\al\mean{A\nabla u+\bg,\vec e_n}=0
\end{cases}
\quad \text{in }B_1^+.    
\end{equation}
Here, the coefficient matrix $A=(a_{ij})_{i,j=1}^n:B_1^+\to\R^{n\times n}$ satisfies the following conditions for some constant $\la>0$
\begin{align}
\label{eq:assump-coeffi}
\begin{cases}
    \text{the uniform ellipticity: }\,\,\,\lambda|\xi|^2\le \mean{A(x)\xi,\xi},\quad\xi\in\R^n,\,\,\,x\in B_1^+, \\
    \text{the uniform boundedness: }\,\,\,|A(x)|\le \frac{1}{\lambda},\quad x\in B_1^+.
\end{cases}
\end{align}
We say that a function $u\in H^{1,\al}(B_1^+)=H^{1}(B_1^+,x_n^\alpha dx)$ is a weak solution of \eqref{eq:deg/sing-pde} if it satisfies
$$\int_{B_1^+}x_n^\al\mean{A\nabla u+\bg,\nabla\phi}=0,$$
for all $\phi\in C^\infty_c(B_1)$. Note that by using a suitable cutoff function and the Cauchy-Schwarz inequality, when $\al\geq1$ it is sufficient to test the equation with $\phi\in C^\infty_c(B_1^+)$.

The main objective of this paper is to establish boundary Schauder type estimates for solutions of \eqref{eq:deg/sing-pde} under Dini mean oscillation (DMO) type conditions on coefficients and data. 
When $\al=0$, it was shown in \cite{DonLeeKim20} that weak solutions of \eqref{eq:deg/sing-pde} are continuously differentiable up to the boundary when coefficients have DMO. Our paper can be primarily viewed as a generalization of \cite{DonLeeKim20} from ``uniformly elliptic'' equations to ``degenerate or singular'' equations. We refer the reader to Theorem \ref{thm:deg-pde-HO} for our main result.

The above class of equations has been studied extensively in the literature. We refer the reader to recent work \cites{SirTerVit21a,SirTerVit21b, TerTorVit22} for H\"older and Schauder estimates and \cites{DP_CVPDE21,DonPha20,DP_TAMS21,DP_JFA} for Sobolev type estimates, as well as the references therein.

In the rest of this subsection, we provide the precise definitions of spaces of DMO functions and $C^{k,\mathrm{DMO}}$ domains.

\begin{definition}[$L^q(d\mu)$-DMO function]\label{def1.1}
Let $\Omega\subseteq\R^n$ be a domain and $f:\Omega\to\R$ a measurable function. Let $\mu$ be a Radon measure, $q\in [1,+\infty)$, $r\in(0,1)$, and
\begin{equation*}
\eta_f^{q,\mu}(r):=\sup_{x_0\in\Omega}\left(\dashint_{\Omega(x_0,r)}|f(x)-\mean{f}^\mu_{\Omega(x_0,r)}|^q\,d\mu(x)\right)^{1/q},
\end{equation*}
where $\Omega(x_0,r):=B_r(x_0)\cap\Omega$ and $\mean{f}^\mu_{\Omega(x_0,r)}=\dashint_{\Omega(x_0,r)}f(x) \,d\mu(x)$. We say that $f$ is of $L^q(d\mu)$ Dini mean oscillations in $\Omega$, briefly $L^q(d\mu)$-DMO in $\Omega$, if $\eta_f^{q,\mu}(r)$ is a Dini function, i.e.,
\begin{equation*}
\int_0^1\frac{\eta_f^{q,\mu}(r)}{r}\, dr<+\infty.
\end{equation*}
\end{definition}
Moreover, the case $q=\infty$ corresponds to uniform Dini continuity; that is, given
\begin{equation*}
\eta_f^{\infty}(r):=\sup_{x_0\in\Omega}\sup_{y,z\in \Omega(x_0,r)}|f(y)-f(z)|,
\end{equation*}
we say that $f$ is Dini continuous in $\Omega$ if $\eta_f^{\infty}(r)$ is a Dini function. Recall the following example given in \cite{DonKim17}. Let $a_{ij}(0)=\delta_{ij}$ and for $0<|x| \le 1/2$,
\[
a_{ij}(x)=\delta_{ij}\left(1+(-\log |x|)^{-\gamma}\right),
\]
where $0<\gamma \le 1$. Then $A$ does not satisfies the $L^\infty$-DMO condition (with respect to the Lebesgue measure) or the H\"older continuity condition.
However, a simple calculation reveals that for any $q\in [1,\infty)$,
\[
\eta^q_A(r)\sim  (-\log r)^{-\gamma-1},
\]
which implies that $A$ satisfies the $L^q$-DMO condition in Definition \ref{def1.1} for any $q\in [1,\infty)$.

In this paper, our primary focus is on $L^1$-DMO type conditions. However, there are situations where we also consider $L^2$-DMO type conditions. Let us remark that any time we work with $L^2$-DMO type conditions, we could work with $L^q$-DMO type conditions for any other $q\in(1,+\infty)$ instead. We decided to choose $q=2$ as the representative of the range $(1,+\infty)$ for the sake of simplicity and clarity.

\begin{definition}[$C^{k,\mathrm{DMO}}$ spaces]\label{CkDMO}
Let $k\in \mathbb N\cup \{0\}$, $\mu$ be a Radon measure, $q\in\{1,2\}$, and $\omega$ a Dini function. We say that $f\in C_{q,\mu}^{k,\omega}(\Omega)$ if $f\in C^k(\overline\Omega)$ and $D^\beta f$ is of $L^q(d\mu)$-DMO in $\Omega$ with DMO modulus $\eta_{D^\beta f}^{q,\mu}\leq C\omega$ for any multiindex $\beta\in(\mathbb N\cup\{0\})^n$ with $|\beta|=\sum_{i=1}^n \beta_i=k$ for some constant $C>0$. Defining the norm
\begin{equation*}
\|f\|_{C_{q,\mu}^{k,\omega}(\Omega)}:=\|f\|_{C^{k}(\Omega)}+\sum_{|\beta|=k}[D^\beta f]_{C_{q,\mu}^{0,\omega}(\Omega)},
\end{equation*}
where
\begin{equation*}
[g]_{C_{q,\mu}^{0,\omega}(\Omega)}:=\sup_{y\in\Omega, r>0}\frac{\left(\dashint_{\Omega\cap B_r(y)}|g(x)-\mean{g}^{\mu}_{\Omega\cap B_r(y)}|^q\,d\mu(x)\right)^{1/q}}{\omega(r)},
\end{equation*}
then $C_{q,\mu}^{k,\omega}(\Omega)$ consists of measurable functions with the finite $C_{q,\mu}^{k,\omega}$-norm.
\end{definition}
Sometimes, we will refer to $C^{k,\mathrm{DMO}}$ functions briefly to indicate functions satisfying the previous definition for certain $q\in\{1,2\}$, $\omega$, and $\mu$. We use $C^{k,q-\mathrm{DMO}}$ as well if we want to indicate the choice of $q\in\{1,2\}$. Given $k\in\mathbb N$, the $k$-th derivative $D^kf$ stands for a generic derivative of order $k$ of the function $f$, i.e., $D^\beta f$ for some $|\beta|=k$. Any time we omit the dependence on the measure, we mean that $\mu$ is the Lebesgue measure of the relevant dimension.

We would like to remark here that $C^k$ uniform Dini spaces embed into $C^k$ spaces with Dini mean oscillation. Moreover, if a radon measure $\mu$ satisfies the doubling property in $\Omega$, then we have

% uniform H\"older spaces with moduli of continuity which are possibly not Dini anymore
\begin{equation*}
    C^{k,\mathrm{Dini}}(\Omega)\subset C^{k,2-\mathrm{DMO}}(\Omega)\subset C^{k,1-\mathrm{DMO}}(\Omega)\subset C^{k}(\overline\Omega).
\end{equation*}
In other words, fixed a Dini function $\omega(r)$ there exists a modulus of continuity $\sigma(r)$ comparable with $\int_0^r\frac{\omega(s)}{s}\, ds$ such that
\begin{equation}\label{chaininclusion}
    C^{k,\omega}(\Omega)\subset C_{2,\mu}^{k,\omega}(\Omega)\subset C_{1,\mu}^{k,\omega}(\Omega)\subset C^{k,\sigma}(\Omega).
\end{equation}
The definition of $C^k$ uniform spaces is standard. In fact, in order to define $C^{k,\omega}(\Omega)$ with $\omega(0)=0$ the reader may just consider the previous Definition \ref{CkDMO} and replace the $L^q(d\mu)$-DMO seminorm with
\begin{equation*}
[g]_{C^{0,\omega}(\Omega)}:=\sup_{x,y\in\Omega, \  x\neq y}\frac{|g(x)-g(y)|}{\omega(|x-y|)}.
\end{equation*}
The first inclusion in \eqref{chaininclusion} holds if $\mu$ is locally finite, the second one follows from the Cauchy-Schwarz inequality, and the last one needs the Lebesgue differentiation theorem and the doubling property; that is, the existence of a positive constant $C>0$ such that
$$\mu(\Omega(x_0,2r))\leq C\mu(\Omega(x_0,r))\qquad\forall x_0\in\Omega, \, 0<r<\mathrm{diam} \,\Omega.$$
We refer to the example below Definition \ref{def1.1} for an $L^q$-DMO function for any $q\in[1,+\infty)$ which is not Dini continuous. We remark that it is still not known whether the second inclusion in \eqref{chaininclusion} is strict.

\begin{definition}[$C^{k,\mathrm{DMO}}$ domains]
Let $n\geq2$, $k\in\mathbb N$, and $q\in\{1,2\}$. A set $\Omega\subseteq\R^n$ is a $C^k$ domain with $L^q$ Dini mean oscillations (briefly $\Omega$ is $C^{k,q-\mathrm{DMO}}$) if it can be locally described as the epigraph of $C_{q}^{k,\omega}$ functions; that is, given $0\in\partial\Omega$, up to rotations and dilations we can parametrize
\begin{equation*}
\Omega\cap B_1=\{x_n>\gamma(x')\},\qquad \partial\Omega\cap B_1=\{x_n=\gamma(x')\}
\end{equation*}
with $\gamma(0)=0$, $\nabla_{x'}\gamma(0)=0$, $\gamma\in C_{q}^{k,\omega}(B_1')$ for some Dini function $\omega$ (i.e. $\eta_{D^k_{x'} \gamma}^{q}\leq C\omega$).
\end{definition}

This DMO type condition on the domain is strictly weaker than the usual Dini type condition on the domain. For example, in two dimensions, consider a domain whose boundary is locally given by 
$$
\gamma(x)=\frac{x}{(-\log|x|)^{1/2}},\quad |x|\le 1/2.
$$
In this case, the domain is $C^{1,\mathrm{DMO}}$, but not $C^{1,\mathrm{Dini}}$.

Although the definition of $C^k$ boundaries with DMO is new in literature, we will show that it is natural and somehow sharp in order to guarantee ``$C^k$ qualitative properties" of solutions, i.e., Schauder type estimates, the Hopf-Oleinik lemma (or boundary point principle), and higher-order boundary Harnack principles.

\subsection{The higher-order boundary Harnack principle}
In this subsection, we present the applications of Schauder type estimates for \eqref{eq:deg/sing-pde}, Theorem \ref{thm:deg-pde-HO}, to higher-order boundary Harnack principles, both on a fixed boundary and across the ``regular" part of the nodal set.

For a bounded domain $\Omega\subset\R^n$, $n\ge2$, we assume that a variable coefficient matrix $A=(a_{ij})_{i,j=1}^n$ is symmetric $A=A^T$ and satisfies \eqref{eq:assump-coeffi} in $\Omega\cap B_1$. Then, the \emph{higher-order boundary Harnack principle} concerns regularity of the ratio of two solutions which vanish on the same fixed boundary, more precisely, two functions $u,v$ weakly solving
\begin{equation}\label{BHconditions}
\begin{cases}
-\div\left(A\nabla v\right)=f &\mathrm{in \ }\Omega\cap B_1,\\
-\div\left(A\nabla u\right)=g &\mathrm{in \ }\Omega\cap B_1,\\
u>0 &\mathrm{in \ }\Omega\cap B_1,\\
u=v=0, \quad \partial_{\nu} u<0&\mathrm{on \ }\partial\Omega\cap B_1,
\end{cases}
\end{equation}
where $0\in\partial\Omega$ and $\nu$ stands for the outward unit normal vector on $\partial\Omega$.

The Schauder $C^{k,\be}$ regularity of the ratio was first established in \cite{DeSSav15} in the case of $C^{k,\be}$ boundaries, $C^{k-1,\be}$ coefficients and right-hand sides, and both in divergence and non-divergence form. See also \cites{BanGar16,Kuk22} for the parabolic counterpart in non-divergence form. Later in \cite{TerTorVit22} a second proof was proposed, based on the following observation: the ratio $w=v/u$, after composing with a straightening diffeomorphism, solves the degenerate equation
\begin{align}\label{eq:BH-deg}
-\div(x_n^2\tilde A\D w)=uf-gv,\quad\text{where }\tilde A=(u/x_n)^2 A.
\end{align}
Then, the regularity of the ratio follows from that of solutions to the degenerate equation \eqref{eq:BH-deg} with the new coefficient $\tilde A$. Finally, using this approach, the work in \cite{JeoVit23} was a first attempt in providing the $C^k$ regularity of the ratio lowering the requirement on boundaries, coefficients, and free terms up to uniform Dini ones.

By following this approach and applying Theorem~\ref{thm:deg-pde-HO}, we achieve the regularity of the ratio under the DMO type conditions imposed on boundaries, coefficients, and data. See Theorem \ref{thm:BHP}.

\medskip

Next, we introduce the boundary Harnack principle across the regular part of the nodal set. Suppose two functions $u,v\in H^1(B_1)$ solve $\div(A\nabla u)=\div(A\nabla v)=0$ in $B_1$ and share their zero sets, i.e., $Z(u)\subseteq Z(v)$, where $Z(u):=u^{-1}\{0\}$. From now on, we will refer to solutions to $\div(A\nabla u)=0$ as $A$-harmonic functions. The study of local regularity of the ratio $w=v/u$ across $Z(u)$ is called \emph{boundary Harnack principle on nodal domains} \cites{LogMal15,LogMal16,LinLin22,TerTorVit22}.

Let us assume that the variable coefficients $A$ are symmetric, satisfy \eqref{eq:assump-coeffi}, and belong to $C_2^{k-1,\omega}(B_1)$ for some $k\in\mathbb N$, $\omega$ a Dini function. Then, by utilizing Theorem~\ref{thm:deg-pde-HO}, we establish Schauder type estimates for the ratio $w=v/u$ across the regular set $R(u):=\{x\in Z(u) \, : \, |\nabla u(x)|\neq0\}$. See Theorem~\ref{thm:BHPRu} for the precise statement.

\subsection{The Hopf-Oleinik boundary point principle}
The Hopf-Oleinik lemma or boundary point principle (BPP) can be extended to elliptic equations in domains with $C^{1,1-\mathrm{DMO}}$ boundary. The BPP follows from \cite{RenSirSoa22}, where the result was proved in case of flat boundaries and $L^1$-DMO coefficients, which is preserved after a standard straightening diffeomorphism \eqref{standard_diffeo} of the $C^{1,1-\mathrm{DMO}}$ boundary. Consequently, under these DMO conditions on \eqref{BHconditions}, whenever $g\ge0$, the condition $\partial_\nu u<0$  holds on $\partial\Omega$. Let us remark here that this is not in contradiction with \cite{ApuNaz16}, where counterexamples to BPP are constructed on boundaries which fail our $C^{1,1-\mathrm{DMO}}$ definition. See Proposition \ref{convexDMO}.

%%%%%%%%%%%%%%%%%%%%%%%%%%%%%%%%%%%%%%%%%%%%%%%%%%%%%%%%%%%%%%%%%%%%%%%%%%%

\subsection{Main results}
We precisely state the main results of this paper.

Our central result is as follows.

\begin{theorem}[Schauder type estimates in the $L^1(x_n^\alpha dx)$-DMO setting]\label{thm:deg-pde-HO}
Let $\alpha>-1$, $k\in\mathbb N$, and $\omega$ be a Dini function. Let $u\in H^{1}(B_1^+,x_n^\alpha dx)$ be a weak solution to \eqref{eq:deg/sing-pde} with $A$ satisfying \eqref{eq:assump-coeffi} in $B_1^+$. Assume that $\bg,A\in C^{k-1}(\overline{B_1^+})$ and $D_{x'}^{k-1}\bg, D_{x'}^{k-1}A\in C_{1,\mu}^{0,\omega}(B_1^+)$ with $d\mu(x)=x_n^\alpha dx$. Then, $u\in C^{k}_\loc(B_1^+\cup B_1')$ and satisfies
\begin{equation}\label{Neumann}
\mean{A\nabla u+\bg,\vec e_n}=0\qquad\mathrm{on \ } B'_1.
\end{equation}
Moreover, if $\|A\|_{C^{k-1}(B_1^+)}+\sum_{|\beta|=k-1}[D_{x'}^{\beta}A]_{C_{1,\mu}^{0,\omega}(B_1^+)}\leq L$, then there exist a modulus $\sigma$ and a constant $C>0$, depending on $n$, $\lambda$, $\omega$, $k$, $\alpha$ and $L$, such that
\begin{equation*}
\|u\|_{C^{k,\sigma}(B_{1/2}^+)}\leq C\left(\|u\|_{L^2(B_1^+,x_n^\alpha dx)}+\|\bg\|_{C^{k-1}(B_1^+)}+\sum_{|\beta|=k-1}[D_{x'}^{\beta}\bg]_{C_{1,\mu}^{0,\omega}(B_1^+)}\right).
\end{equation*}
\end{theorem}
A similar result was attained in \cite{TerTorVit22} when coefficients belong to the H\"older space $C^{k-1,\al}$. However, as the argument in \cite{TerTorVit22} is specialized for the homogeneous power-type moduli, it is not applicable to our setting. Instead, we adopt Campanato's approach and utilize the weak type-(1,1) estimate presented in \cite{DonLeeKim20}, following the original idea in \cites{DonKim17, DonEscKim18}. \cite{DonLeeKim20} deals with the specific case of Theorem~\ref{thm:deg-pde-HO} with $\al=0$ and $k=1$, and its crucial step involves the growth estimate of 
$$
\vp(x_0,r):=\inf_{\bq\in\R^n}\left(\dashint_{B_r(x_0)\cap B_1^+}|\D u-\bq|^p\right)^{1/p},\quad  0<r<1/2,\,x_0\in B_{1/2}^+,\, 0<p<1.
$$
However, it turns out that in our weighted context, we have to work with
\begin{align*}
    \psi(x_0,r):=\begin{cases}
        \inf_{\bq\in\R^n}\left(\dashint_{B_r(x_0)}|\mean{\D_{x'}u,U^{x_0}}-\bq|^pd\mu\right)^{1/p},& 0<r\le (x_0)_n/2,\\
        \inf_{\bq'\in\R^{n-1}}\left(\dashint_{B_r(x_0)\cap B_4^+}|\mean{\D_{x'}u-\bq',U}|^pd\mu\right)^{1/p},& (x_0)_n/2< r<1/2,
    \end{cases}
\end{align*}
where $U:=\mean{A\D u+\bg,\vec{e}_n}$ and $U^{x_0}:=(x_n/(x_0)_n)^\al\mean{A\D u+\bg,\vec{e}_n}$. This is where our proof substantially differs from \cite{DonLeeKim20}, requiring careful handling of the discontinuity of $\psi$ at $r=(x_0)_n/2$.

\medskip
Next, we present higher-order boundary Harnack principle on a fixed boundary.

\begin{theorem}[Higher-order boundary Harnack principle]\label{thm:BHP}
Let $k\in\mathbb N$, and $\omega$ a Dini function. Consider two functions $u,v\in H^1(\Omega\cap B_1)$ solving \eqref{BHconditions} with $A$ symmetric and satisfying \eqref{eq:assump-coeffi}. Assume that $A,f,g\in C_1^{k-1,\omega}(\Omega\cap B_1)$ and $\gamma\in C_1^{k,\omega}(B'_1)$. Then, $w=v/u$ belongs to $C^{k}_\loc(\overline{\Omega}\cap B_1)$ and satisfies the following boundary condition
\begin{equation*}
2\mean{\nabla u,\nu}\mean{A\nabla w,\nu}+f-gw=0\qquad\mathrm{on \ } \partial\Omega\cap B_1.
\end{equation*}
Moreover, if $\|A\|_{C_1^{k-1,\omega}(\Omega\cap B_{1})}+\|\gamma\|_{C_1^{k,\omega}(B_1')}+\|g\|_{C_1^{k-1,\omega}(\Omega\cap B_{1})}\le L_1, \|u\|_{L^2(\Omega\cap B_{1})}\le L_2$, and $\inf_{\partial\Omega\cap B_{3/4}}|\partial_{\nu}u|\ge L_3>0$, then the following estimate holds true
\begin{equation*}\label{eq.daperfez}
\left\|\frac{v}{u}\right\|_{C^{k,\sigma}(\Omega\cap B_{1/2})}\le C\left(\left\|v\right\|_{L^2(\Omega\cap B_{1})}+\|f\|_{C_1^{k-1,\omega}(\Omega\cap B_{1})}\right)
\end{equation*}
with a modulus of continuity $\sigma$ and a positive constant $C$ depending on $n$, $\lambda$, $\omega$, $k$, $L_1$, $L_2$, and $L_3$. Finally, if $u(\vec e_n/2)=1$ and $v>0$ in $\Omega\cap B_{1}$, then
\begin{equation*}
\left\|\frac{v}{u}\right\|_{C^{k,\sigma}(\Omega\cap B_{1/2})}\le C\left(\left|\frac{v}{u}(\vec e_n/2)\right|+\|f\|_{C_1^{k-1,\omega}(\Omega\cap B_{1})}\right)
\end{equation*}
with $\sigma$ and a positive constant $C$ depending only on $n$, $\lambda$, $\omega$, $k$, $L_1$, and $L_3$.
\end{theorem}

Similar to the approaches in H\"older and uniform Dini settings \cites{TerTorVit22, JeoVit23}, we establish Theorem~\ref{thm:BHP} by reducing it to the Schauder type estimate for the degenerate equation \eqref{eq:BH-deg}. As $u/x_n$ is absorbed into the new coefficient $\tilde A$, the regularity property of $u/x_n$ is crucial for applying the Schauder type estimate, Theorem~\ref{thm:deg-pde-HO}. While this is rather immediate in \cites{TerTorVit22, JeoVit23}, in our DMO context, it is highly nontrivial to verify that $u/x_n$ is in the DMO-space. We prove this by employing Campanato's approach in a clever manner; see Proposition~\ref{prop:ratio-L1-DMO}.

% Then, we would like to remark that in the $L^2$-DMO setting we will derive a higher order boundary Harnack principle which comes with DMO type information on derivatives of the ratio $w=v/u$, see Theorem \ref{thm:BHP2}.

% {\color{blue}
% Let us remark here that the $L^2$-DMO assumptions we make on coefficients ensure local $C_2^{k,\overline\omega}$ regularity of solutions $u$, with $\overline\omega$ a Dini function, by Theorem \ref{CkUnif2}. Then, by proving the regularity of $w$ from both sides of $R(u)$ together with a gluing lemma, one can get
% }

\medskip
Lastly, we state the boundary Harnack principle across the regular zero set, subject to the $L^2$-DMO type condition on the coefficient. We recall the regular set $R(u)=\{x\in Z(u)\,:\, |\D u(x)|\neq0\}$ and denote the singular set by $S(u)=\{x\in Z(u)\,:\, |\D u(x)|=0\}$.

\begin{theorem}[Schauder type estimates for the ratio across regular zero sets]\label{thm:BHPRu}
Let $k\in\mathbb N$ and $\omega$ a Dini function. Consider two $A$-harmonic functions $u,v\in H^1(B_1)$ with $A$ symmetric and satisfying \eqref{eq:assump-coeffi} in $B_1$, $A\in C_2^{k-1,\omega}(B_1)$. Assume that $S(u)\cap B_1=\emptyset$ and $Z(u)\subseteq Z(v)$. Then, $w=v/u$ belongs to $C^{k}_\loc(B_1)$ and satisfies the following boundary condition
\begin{equation*}
\mean{A\nabla w,\nu}=0\qquad\mathrm{on \ } R(u)\cap B_1,
\end{equation*}
where $\nu$ is the unit normal vector on $R(u)$. Moreover, if $\|A\|_{C_2^{k-1,\omega}(B_1)}\leq L_1$ then
\begin{equation*}
\left\| \frac{v}{u} \right\|_{C^{k,\sigma}(B_{1/2})}\leq C\left\|v\right\|_{L^2(B_1)},
\end{equation*}
with modulus $\sigma$ and a positive constant $C$ depending on $n,\lambda,\omega,L_1,u$, and its nodal set $Z(u)$.
\end{theorem}

For the proof, we first show that the solution $u$ belongs to the DMO-space $C^{k,\hat\omega}_2$ under the $L^2$-DMO assumption on the coefficient. Then, we flatten the zero-set of $u$ and reduce the problem to the degenerate equation with a DMO-type coefficient. Finally, we apply Theorem~\ref{thm:deg-pde-HO} on both sides of $R(u)$ together with a gluing lemma. 

It is noteworthy that the standard straightening diffeomorphism \eqref{standard_diffeo} does not ensure that the new coefficient retains the DMO property, as it is not preserved when restricting to lower-dimensional subsets, unlike the H\"older condition. To address this issue, we employ a different type of diffeomorphism to flatten the level set; see \eqref{udiffeo}.

%%%%%%%%%%%%%%%%%%%%%%%%%%%%%%%%%%%%%%%%%%%%%%%%%%%%%%%%

\subsection{Notation and structure of the paper}
We use the following notation in this paper.
\begin{itemize}
\item $\R^n$ stands for the $n$-dimensional Euclidean space. We indicate the points in $\R^n$ by $x=(x',x_n)$, where $x'=(x_1,\cdots,x_{n-1})\in \R^{n-1}$, and identify $\R^{n-1}$ with $\R^{n-1}\times\{0\}$.

For $x\in\R^n$ and $r>0$, we let
\begin{alignat*}{2}
  B_r(x)&=\{y\in \R^n:|x-y|<r\},&\quad&\text{ball in $\R^n$,}\\
  B^{\pm}_r(x')&=B_r(x',0)\cap \{\pm y_n>0\},&& \text{half-ball},\\
  B'_r(x')&=B_r(x',0)\cap \{y_n=0\}, &&\text{thin ball.}
\end{alignat*}
When the center is the origin, we simply write $B_r=B_r(0)$, $B_r^\pm=B_r^\pm(0)$, and  $B_r'=B'_r(0)$.

\item The notation $\mean{\cdot,\cdot}$ stands for the scalar product of two vectors; that is, for $x=(x_1,\ldots,x_n)$ and $y=(y_1,\ldots,y_n)$, then $\mean{x,y}=\sum_{i=1}^nx_iy_i$. The orthonormal basis of $\R^n$ is denoted by $\vec e_i$, $i=1,\ldots,n$. %; that is $\mean{x,\vec e_i}=x_i$.
\item We denote the set of positive integers by $\mathbb{N}=\{1,2,3,\cdots\}$.
\item Let $\beta=(\beta_1,\ldots,\beta_n)\in(\mathbb N\cup\{0\})^n$ be a multiindex. Given $|\beta|=\sum_{i=1}^n\beta_i$ and $\partial x^\beta=\prod_{i=1}^n\partial x_i^{\beta_i}$,  the $\beta$ partial derivative of order $|\beta|$ is given by
$$D^\beta u=\frac{\partial^{|\beta|}u}{\partial x^\beta}.$$
When we write $D^ku$ for some $k\in\mathbb N$, we mean a generic partial derivative $D^\beta u$ with $|\beta|=k$. By $D^k_{(x_{i_1},x_{i_2},\ldots,x_{i_k})}u$, we mean we consider only derivatives of $u$ of order $k$ with respect to some chosen directions $\vec e_{i_j}$ with $i_j\in\{i_1,\ldots,i_k\}\subset\{1,\ldots,n\}$.
\end{itemize}

The rest of the paper is organized as follows: Section \ref{3} is devoted to the proof of Theorem \ref{thm:deg-pde-HO}, i.e., Schauder type estimates for degenerate or singular equations as in \eqref{eq:deg/sing-pde} for general powers $\alpha>-1$. In Section~\ref{sec:Hopf}, we discuss the validity of the boundary point principle on $C^{1,1-\text{DMO}}$ boundaries. In Section \ref{sec:BHP}, we prove Theorems \ref{thm:BHP} and \ref{thm:BHPRu}, i.e., the boundary Harnack principle on a fixed $C^{k,1-\mathrm{DMO}}$ boundary and the boundary Harnack principle across regular zero sets of solutions to elliptic equations with $C^{k-1,2-\mathrm{DMO}}$ variable coefficients. Finally, we prove some properties of DMO-functions in Appendix \ref{A}.

\section{Schauder type estimates for degenerate or singular equations}\label{3}
This section is devoted to the proof of our main result, Theorem \ref{thm:deg-pde-HO}. Throughout this section, we fix
$$
d\mu=x_n^\al dx, \quad\al>-1.
$$

\subsection{Weak type-\texorpdfstring{$(1,1)$}{} estimates}
In this subsection, we prove the following version of weak type-$(1,1)$ estimates for the solution of degenerate or singular equations, which will play a significant role in the proof of the $C^1$-estimate in Theorem~\ref{thm:deg-pde}.

\begin{lemma}
    \label{lem:weak-type-(1,1)}
Let $\bar A$ be a constant matrix satisfying
$$
\la|\xi|^2\le \mean{\bar A\xi,\xi},\quad \xi\in\R^n,\qquad\mathrm{and}\qquad|\bar A|\le 1/\la
$$
for some constant $\la>0$. Let $\cD$ and $\tilde\cD$ be smooth and convex domains in $\R^n_+$ with $B_{1}^+\subset\cD\subset B^+_{4/3}$ and $B^+_{3/2}\subset\tilde\cD\subset B_2^+$. For $\bff\in H^{1,\al}(\tilde\cD)$, let $u\in H^{1,\al}(\tilde\cD)$ be a weak solution of
$$
-\div(x_n^\al\bar A\D u)=\div(x_n^\al\bff\chi_\cD)\quad\text{in }\tilde\cD,
$$
with the conormal boundary condition
$$
x_n^\al\mean{\bar A\D u+\bff\chi_{\cD},\nu}=0\quad\text{on }\partial \tilde\cD.
$$
Then there exists a constant $C=C(n,\la,\al)>0$ such that for any $t>0$
$$
\mu(\{x\in \cD\,:\,|\D u(x)|>t\})\le \frac{C}{t}\int_\cD|\bff|d\mu.
$$
\end{lemma}

The proof of Lemma~\ref{lem:weak-type-(1,1)} relies on the following auxiliary results.

\begin{lemma}
    \label{lem:weak-type-(1,1)-aux}
Let $\bar A$, $\cD$ and $\tilde\cD$ be as in Lemma~\ref{lem:weak-type-(1,1)}. Given $x_0\in \cD$ and $0<r<\frac12 \diam \cD$, suppose that $\bb\in L^2(\cD;\R^n,d\mu)$ is supported in $B_r(x_0)\cap \cD$ and satisfies $\int_{B_r(x_0)\cap \cD}\bb\, d\mu=\mathbf{0}$. If $\tilde u$ is a solution of
\begin{align}\label{eq:weak-sol-u}
-\div(x_n^\al\bar A\D\tilde u)=\div(x_n^\al\bb)\quad\text{in }\tilde\cD,\qquad x_n^\al\mean{\bar A\D\tilde u+\bb,\nu}=0\quad\text{on }\partial \tilde\cD,
\end{align}
then there exists a constant $C>0$, depending only on $n,\al,\la$, such that
$$
\int_{\cD\setminus B_{2r}(x_0)}|\D\tilde u|d\mu\le C\int_{B_r(x_0)\cap \cD}|\bb|d\mu.
$$
\end{lemma}

\begin{proof}
Since the proof of this lemma follows a portion of the proof in \cite{DonLeeKim20}*{Lemma~2.12} with straightforward modifications, we shall provide the outline of the proofs instead of going into the details.

For any $R\ge 2r$ such that $\cD\setminus B_R(x_0)\neq\emptyset$ and a function $\bg\in C_c^\infty((B_{2R}(x_0)\setminus B_R(x_0))\cap \cD;\R^n)$, let $\tilde v\in H^{1,\al}(\tilde\cD)$ be a solution to
\begin{align}\label{eq:weak-sol-v}
-\div(x_n^\al\bar A^T\D\tilde v)=\div(x_n^\al\bg)\quad\text{in }\tilde\cD,\qquad x_n^\al\mean{\bar A^T\D\tilde v+\bg,\nu}=0\quad\text{on }\partial \tilde\cD
\end{align}
By testing \eqref{eq:weak-sol-v} with $\tilde u$ and \eqref{eq:weak-sol-u} with $\tilde v$ and using the assumptions on $\bb$, we obtain
\begin{align*}
    \int_\cD\mean{\D\tilde u,\bg}d\mu=\int_\cD\mean{\D\tilde v,\bb}d\mu=\int_{B_r(x_0)\cap \cD}\mean{\D\tilde v-\mean{\D\tilde v}^\mu_{B_r(x_0)\cap \cD},\bb}d\mu.
\end{align*}
This gives
\begin{align*}
    &\left|\int_{(B_{2R}(x_0)\setminus B_R(x_0))\cap \cD}\mean{\D\tilde u,\bg}d\mu\right|\le 2r\|D^2\tilde v\|_{L^\infty(B_r(x_0)\cap \cD)}\int_{B_r(x_0)\cap \cD}|\bb|d\mu.
\end{align*}
Since $\tilde v$ solves the homogeneous equation $\div(x_n^\al\bar A^T\D\tilde v)=0$ in $B_R(x_0)\cap \cD$ with the homogeneous boundary condition and $r\le R/2$, we have by \cite{DonPha20}*{Proposition~4.4}
\begin{align*}
    \|D^2\tilde v\|_{L^\infty(B_{r}(x_0)\cap \cD)}^2&\le \frac{C}{R^2}\dashint_{B_R(x_0)\cap \tilde\cD}|\D\tilde v|^2d\mu\le \frac{C}{R^2\mu(B_R(x_0)\cap\tilde\cD)}\int_{\tilde\cD}|\bg|^2d\mu.
\end{align*}
Combining the previous two estimates yields
\begin{align*}
    &\left|\int_{(B_{2R}(x_0)\setminus B_R(x_0))\cap \cD}\mean{\D\tilde u,\bg}d\mu\right|\\
    &\le \frac{Cr}{R[\mu(B_R(x_0)\cap\tilde\cD)]^{1/2}}\int_{B_r(x_0)\cap \cD}|\bb|d\mu\cdot\|\bg\|_{L^2((B_{2R}(x_0)\setminus B_R(x_0))\cap \cD,d\mu)}.
\end{align*}
By using the duality and applying H\"older's inequality, we deduce
$$
\|\D\tilde u\|_{L^1((B_{2R}(x_0)\setminus B_R(x_0))\cap \cD,d\mu)}\le \frac{Cr}R\int_{B_r(x_0)\cap \cD}|\bb|d\mu.
$$
Taking the smallest positive integer $K$ such that $\cD\subset B_{2^{K+1}r}(x_0)$ and applying the above inequality with $R=2r,2^2r,\ldots,2^Kr$, we conclude
\begin{equation*}
    \int_{\cD\setminus B_{2r}(x_0)}|\D \tilde u|d\mu\le Cr\int_{B_r(x_0)\cap \cD}|\bb|d\mu\sum_{k=1}^K\frac1{2^kr}\le C\int_{B_r(x_0)\cap \cD}|\bb|d\mu.\qedhere
\end{equation*}
\end{proof}

\begin{lemma}
    \label{lem:weak-type-(1,1)-aux-2}
Let $\cD\subset B_1^+$ be a smooth convex domain and $T$ be a bounded linear operator from $L^2(\cD;\R^n,d\mu)$ to $L^2(\cD;\R^n,d\mu)$. Suppose that there exists a constant $C_0>0$ such that for any $x_0\in \cD$ and $0<r<\frac12\diam \cD$,
$$
\int_{\cD\setminus B_{2r}(x_0)}|T\bb|\le C_0\int_{B_r(x_0)\cap \cD}|\bb|
$$
whenever $\bb\in L^2(\cD;\R^n,d\mu)$ is supported in $B_r(x_0)\cap \cD$ and satisfies $\int_\cD\bb\, d\mu=\mathbf{0}$. Then for $\bff\in L^2(\cD;\R^n,d\mu)$ and any $t>0$, it holds that
$$
\mu\left(\{x\in \cD\,:\, |T\bff(x)|>t\}\right)\le \frac{C}{t}\int_\cD|\bff|d\mu,
$$
where $C>0$ is a constant, depending only on $n,\al,\la,\cD,C_0$.
\end{lemma}

\begin{proof}
In view of Lemma~\ref{lem:doubl}, $\cD$ equipped with the standard Euclidean metric and the weighted measure $\mu$ (restricted to $\cD$) is a space of homogeneous type. Moreover, the Lebesgue differentiation theorem is available in our situation. With these properties at hand, we can follow the proof of \cite{DonEscKim18}*{Lemma~4.1}, making obvious adjustments as needed to conclude the lemma.
\end{proof}

Now, Lemma~\ref{lem:weak-type-(1,1)} can be directly derived from the preceding two lemmas.

\begin{proof}[Proof of Lemma~\ref{lem:weak-type-(1,1)}]
Given any $\bff\in H^{1,\al}(\tilde\cD)$, we solve for $u$. As $u$ is unique up to a constant, the map $T:\bff\longmapsto \D u$ is well defined, and it is a bounded linear operator on $L^2(\cD,d\mu)$. Thus, Lemma~\ref{lem:weak-type-(1,1)} follows by combining Lemmas~\ref{lem:weak-type-(1,1)-aux} and ~\ref{lem:weak-type-(1,1)-aux-2}.
\end{proof}

%%%%%%%%%%%%%%%%%%%%%%%%%%%%%%%%%%%%%%%%%%%%%%%%%%

\subsection{\texorpdfstring{$C^1$}{} estimates}
In this subsection, we establish Theorem~\ref{thm:deg-pde-HO} for the case $k=1$. We will subsequently generalize this result to the cases $k\ge1$ by employing an induction argument in the next subsection.

\begin{theorem}
    \label{thm:deg-pde}
For $\al>-1$, let $u\in H^{1,\al}(B_4^+)$ be a weak solution of
\begin{align}\label{eq:deg-pde}
-\div(x_n^\al A\D u)=\div(x_n^\al\bg)\quad\text{in }B_4^+,
\end{align}
satisfying the conormal boundary condition on $B_4'$
\begin{align*}
\lim_{x_n\to0^+}x_n^\al\mean{A\D u+\bg,\vec{e}_n}=0.
\end{align*}
Suppose $A=[a_{ij}]_{n\times n}$ satisfies \eqref{eq:assump-coeffi} in $B_4^+$. If $A$ and $\bg$ are of $L^1(d\mu)$-DMO in $B_4^+$, then $u\in C^1(\overline{B_1^+})$.
\end{theorem}

Note that the case $\al=0$ was achieved in \cite{DonLeeKim20}*{Proposition~3.2}. A similar result can be found in \cite{DonEscKim18}*{Proposition~2.7}, where the zero Dirichlet boundary condition on $B_4'$ was imposed instead of the conormal boundary condition. In our proof of Theorem~\ref{thm:deg-pde}, we follow the lines in these two papers. However, our case involving $\al>-1$ requires significantly more intricate technical considerations due to the degenerate or singular nature of the problem.

As in \cites{DonEscKim18, DonLeeKim20}, we will establish an a priori estimate of the modulus of continuity of $\D u$ under the assumption $u\in C^1(\overline{B_3^+})$. The general case can be obtained by a standard approximation argument (see e.g. pages 134-135 in \cite{Don12}).

Given $x_0=(x_0',(x_0)_n)\in \overline{B_3^+}$, we put $d_{x_0}:=\dist(x_0,B'_3)=(x_0)_n\ge0$. We set
\begin{align*}
    U:=\mean{A\D u+\bg,\vec{e}_n},\quad\text{and}\quad U^{x_0}:=(x_n/d_{x_0})^\al\mean{A\D u+\bg,\vec{e}_n}\,\text{ when }d_{x_0}>0.
\end{align*}
We fix $0<p<1$ and consider
\begin{align*}
    \psi(x_0,r):=\begin{cases}
        \inf_{\bq\in\R^n}\left(\dashint_{B_r(x_0)}|\mean{\D_{x'}u,U^{x_0}}-\bq|^pd\mu\right)^{1/p},& 0<r\le d_{x_0}/2,\\
        \inf_{\bq'\in\R^{n-1}}\left(\dashint_{B_r(x_0)\cap B_4^+}|\mean{\D_{x'}u-\bq',U}|^pd\mu\right)^{1/p},& d_{x_0}/2< r<1/2.
    \end{cases}
\end{align*}
We would like to mention that for the latter case $d_{x_0}/2<r<1/2$, the infimum can be taken over $\bq\in\R^n$ instead of $\bq'\in\R^{n-1}$. However, we opt for this formulation of $\psi$ as it readily yields the boundary condition $U=0$ (equation \eqref{Neumann}), which is a significant ingredient in the boundary Harnack Principe across the regular set; see Subsection~\ref{5}.

\medskip

Next, we introduce several Dini functions derived from $\eta^{1,\mu}_\bullet$ to be used in this section, where $\bullet$ represents either $A$ or $\bg$. We remark that $r\longmapsto \eta^{1,\mu}_\bullet(r)$ is not necessarily nondecreasing. Instead, there are constants $C>c>0$, depending only on $n$ and $\al$, such that
\begin{align}
    \label{eq:alm-mon}
    c\eta^{1,\mu}_\bullet(r)\le \eta^{1,\mu}_\bullet(s)\le C\eta^{1,\mu}_\bullet(r)
\end{align}
whenever $r/2\le s\le r<1$. See \cite{Li17}.

Given a constant $0<\ka<1$, we define
\begin{align}\label{eq:tilde-eta}
\tilde\eta_\bullet^{1,\mu}(r):=\sum_{i=0}^\infty\ka^{i/2}\left(\eta^{1,\mu}_\bullet(\ka^{-i}r)[\ka^{-i}r\le 1]+\eta^{1,\mu}_\bullet(1)[\ka^{-i}r>1]\right),\quad 0<r<1,
\end{align}
where $[\cdot]$ is Iverson's bracket notation (i.e., $[P]=1$ when $P$ is true, while $[P]=0$ otherwise). We note that $\tilde\eta_\bullet^{1,\mu}\ge \eta_\bullet^{1,\mu}$ and $\tilde\eta_\bullet^{1,\mu}$ is a Dini function satisfying \eqref{eq:alm-mon}. See e.g., \cite{Don12}. Given $0<\be<1$, we also set
\begin{align}
    \label{eq:hat-eta}
    \hat\eta_\bullet^{1,\mu}(r):=\sup_{\rho\in [r,1)}(r/\rho)^{\be}\tilde\eta_\bullet^{1,\mu}(\rho),\quad 0<r<1.
\end{align}
Note that $\tilde\eta^{1,\mu}_\bullet$ is a Dini function satisfying \eqref{eq:alm-mon} and that $\hat\eta_\bullet^{1,\mu}\ge\tilde\eta_\bullet^{1,\mu}$ and $r\longmapsto \frac{\hat\eta_\bullet^{1,\mu}(r)}{r^{\be}}$ is nonincreasing. See e.g., \cites{Don12, JeoVit23}.

\medskip
Now, we start the proof of Theorem~\ref{thm:deg-pde} by first addressing the case when the center of the ball lies on $B_3'$.

In this subsection $u$ is a solution of \eqref{eq:deg-pde}, where both $A$ and $\bg$ are of $L^1(d\mu)$-DMO, as in Theorem~\ref{thm:deg-pde}. We write for simplicity $\|\bg\|_\infty=\|\bg\|_{L^\infty(B_4^+)}$.

\begin{lemma}\label{lem:deg-bdry}
     Let $\bar x_0\in B'_3$, $0<\be<1$ and $0<p<1$. Then, for any $0<\rho<r<1/2$, it holds that
    \begin{align}\label{eq:deg-phi-est-bdry}\begin{split}
    \psi(\bar x_0,\rho)&\le C(\rho/r)^{\be}\psi(\bar x_0,r)+C\|\D u\|_{L^\infty(B^+_{2r}(\bar x_0))}\tilde\eta_A^{1,\mu}(2\rho)+C\tilde\eta_\bg^{1,\mu}(2\rho),
    \end{split}\end{align}
    where $C=C(n,\la,\al,p,\be)>0$ are constants and $\tilde\eta_\bullet^{1,\mu}$ is as in \eqref{eq:tilde-eta}.
\end{lemma}

\begin{proof}
Without loss of generality, we may assume $\bar x_0=0$. We fix $0<r<1/2$ and write for simplicity $\bar A=\mean{A}^\mu_{B^+_{2r}}$ and $\bar\bg=\mean{\bg}^\mu_{B^+_{2r}}$. Note that $u$ satisfies the following equation
$$
-\div(x_n^\al\bar A\D u)=\div(x_n^\al\bg)+\div(x_n^\al(A-\bar A)\D u).
$$
We set
$$
\hat u(x):=u(x)+\bar a_{nn}^{-1}\bar g_nx_n,
$$
and observe that it solves
\begin{align*}
  \begin{cases}
    -\div(x_n^\al\bar A\D\hat u)=\div(x_n^\al((A-\bar A)\D u+\bg-\bar\bg))&\text{in }B^+_{2r},\\
    \lim_{x_n\to0+}x_n^\al\mean{\bar A\D\hat u+(A-\bar A)\D u+\bg-\bar\bg,\vec{e}_n}=0&\text{on }B'_{2r}.
\end{cases}\end{align*}
We take a smooth and convex domains $\cD_r$ and $\tilde\cD_r$ such that $B_r^+\subset\cD_r\subset B_{\frac43r}^+$and $B^+_{\frac32r}\subset\tilde\cD_r\subset B^+_{2r}$, and write $\cD'_r=\overline{\cD_r}\cap B'_4$ and $\tilde\cD'_r=\overline{\tilde\cD_r}\cap B'_4$. We decompose $\hat u=\hat v+\hat w$, where $\hat w$ is a solution of
$$
-\div(x_n^\al\bar A\D\hat w)=\div\left(x_n^\al\left((A-\bar A)\D u+\bg-\bar\bg\right)\chi_{\cD_r}\right)\quad\text{in }\tilde\cD_{r}
$$
with the boundary condition
$$
x_n^\al\mean{\bar A\D\hat w+((A-\bar A)\D u+\bg-\bar\bg)\chi_{\cD_r},\nu}=0\quad\text{on }\partial\tilde\cD_{r}.
$$
For any $t>0$, we have by applying Lemma~\ref{lem:weak-type-(1,1)} with scaling
$$
\mu(\{x\in B_r^+\,:\,|\D\hat w(x)|>t\})\le \frac{C(n,\la,\al)}t\int_{B_{2r}^+}|(A-\bar A)\D u+\bg-\bar\bg|d\mu.
$$
This inequality implies that for any $\tau\in(0,+\infty)$
\begin{align*}
    &\int_{B_r^+}|\D\hat w|^pd\mu=\int_0^{+\infty} pt^{p-1}\mu(\{x\in B_r^+\,:\, |\D\hat w(x)|>t\})\,dt\\
    &\le \int_0^\tau pt^{p-1}\mu(B_r^+)\,dt+\int_\tau^{+\infty} pt^{p-1}\left(\frac{C}t\int_{B_{2r}^+}|(A-\bar A)\D u+\bg-\bar\bg|d\mu\right)dt\\
    &\le C\mu(B_r^+)\tau^p+C\left(\int_{B_{2r}^+}|(A-\bar A)\D u+\bg-\bar\bg|d\mu\right)\tau^{p-1}.
\end{align*}
Taking $\tau=\frac{\int_{B_{2r}^+}|(A-\bar A)\D u+\bg-\bar\bg|d\mu}{\mu(B_r^+)}$, we get
$$
\int_{B_r^+}|\D\hat w|^pd\mu\le C\mu(B_r^+)^{1-p}\left(\int_{B_{2r}^+}|(A-\bar A)\D u+\bg-\bar\bg|d\mu\right)^p.
$$
It follows that
\begin{align}
    \label{eq:deg-sol-hom-diff-est}
    \begin{split}
        \left(\dashint_{B_r^+}|\D\hat w|^pd\mu\right)^{1/p}&\le C\mu(B_r^+)^{-1}\int_{B_{2r}^+}|(A-\bar A)\D u+\bg-\bar\bg|d\mu\\
        %&\le C\dashint_{B_{2r}^+}|(A-\bar A)\D u+\bg-\bar\bg|d\mu\\
        &\le C\|\D u\|_{L^\infty(B_{2r}^+)}\eta_A^{1,\mu}(2r)+C\eta_\bg^{1,\mu}(2r).
    \end{split}
\end{align}

Next, we observe that $\hat v=\hat u-\hat w$ satisfies
\begin{align}\label{eq:hom-repla-pde}\begin{cases}
\div(x_n^\al\bar A\D\hat v)=0&\text{in }\cD_{r},\\
\lim_{x_n\to0^+}x_n^\al\mean{\bar A\D\hat v,\vec{e}_n}=0&\text{on }\cD'_r.
\end{cases}\end{align}
Since $D_i\hat v$ satisfies the same equation for $1\le i\le n-1$, we have by applying \cite{DonPha20}*{Lemma~4.2, Proposition~4.4} with a standard iteration
\begin{align}\label{eq:hom-partial-est}
\|D D_i\hat v\|_{L^\infty(B_{r/2}^+)}\le \frac{C}{r}\left(\dashint_{B_r^+}|D_i\hat v|^pd\mu\right)^{1/p}\le \frac{C}r\left(\dashint_{B_r^+}|\D_{x'}\hat v|^pd\mu\right)^{1/p},
\end{align}
where $\D_{x'}\hat v=(\partial_{x_1}\hat v,\cdots,\partial_{x_{n-1}}\hat v)$. To obtain the similar estimate for $D_{nn}\hat v$, we use an idea in \cite{DonPha20}. Denote
$$\hat V:=\mean{\bar A\D\hat v,\vec{e}_n}=\sum_{j=1}^n\bar{a}_{nj}D_j\hat v.$$
From the equation $\div(x_n^\al\bar A\D \hat v)=0$ in $\cD_{r}$, we infer
$$
\partial_n\left(x_n^\al\hat V\right)=-x_n^\al\sum_{i=1}^{n-1}\sum_{j=1}^n\bar a_{ij}D_{ij}\hat v.
$$
By combining this with $\lim_{x_n\to0^+}x_n^\al\hat V(x)=0$ and \eqref{eq:hom-partial-est}, we deduce that for any $x\in B_{r/2}^+$
\begin{align}
    \label{eq:hom-n-deriv-est}
    \begin{split}
        \left|x_n^\al\hat V(x)\right|&=\left|\int_0^{x_n}\partial_s(s^\al\hat V(x',s))ds\right|\le \int_0^{x_n}s^\al\sum_{i=1}^{n-1}\sum_{j=1}^n|\bar a_{ij}D_{ij}\hat v(x',s)|ds\\
        &\le \frac{C}r\left(\dashint_{B_r^+}|\D_{x'}\hat v|^pd\mu\right)^{1/p}\int_0^{x_n}s^\al\,ds\le \frac{Cx_n^{\al+1}}r\left(\dashint_{B_r^+}|\D_{x'}\hat v|^pd\mu\right)^{1/p}.
    \end{split}
\end{align}
This gives
$$
|\hat V(x)|\le \frac{Cx_n}r\left(\dashint_{B_r^+}|\D_{x'}\hat v|^pd\mu\right)^{1/p},\quad x\in B^+_{r/2}.
$$
Combining this with \eqref{eq:hom-partial-est} gives that for small $\ka\in(0,1/2)$ to be chosen later 
\begin{align}
    \label{eq:hom-repla}\begin{split}
    &\left(\dashint_{B_{\ka r}^+}\left(|\D_{x'}\hat v-\mean{\D_{x'}\hat v}_{B_{\ka r}^+}|^p+|\hat V|^p\right)d\mu\right)^{1/p}\\
    &\le C\ka \left(\dashint_{B_r^+}|\D \hat v|^pd\mu\right)^{1/p}\le C\ka\left(\dashint_{B_r^+}\left(|\D_{x'}\hat v|^p+|\hat V|^p\right)d\mu\right)^{1/p}.
\end{split}\end{align}
For any constant vector $\bq'=(q_1,\ldots q_{n-1})\in \R^{n-1}$, we set
$$
\tilde v(X):=\hat v(X)-\mean{\bq',x'}+\int_0^{x_n}[\bar a_{nn}(y_n)]^{-1}\left(\sum_{i=1}^{n-1}\bar a_{ni}(y_n)q_i\right)dy_n$$
and
$$\tilde V:=\mean{\bar A\D\tilde v,\vec{e}_n}.
$$
It is easily seen that $\tilde v$ satisfies \eqref{eq:hom-repla-pde} and $\tilde V=\hat V$ in $B_r^+$. This enables us to replace $\hat v$ and $\hat V$ with $\tilde v$ and $\tilde V$ in \eqref{eq:hom-repla}, respectively, to have
\begin{align}
    \label{eq:hom-repla-iter}
    \begin{split}&\left(\dashint_{B_{\ka r}^+}\left(|\D_{x'}\hat v-\mean{\D_{x'}\hat v}_{B_{\ka r}^+}|^p+|\hat V|^p\right)d\mu\right)^{1/p}\\
    &\le C\ka\left(\dashint_{B_r^+}\left(|\D_{x'}\hat v-\bq'|^p+|\hat V|^p\right)d\mu\right)^{1/p}.
\end{split}\end{align}
Let $\hat U:=\mean{\bar A\D \hat u,\vec{e}_n}$. By using $\hat U=\hat V+\mean{\bar A\D \hat w,\vec{e}_n}$ and \eqref{eq:hom-repla-iter}, we get
\begin{align}
    \label{eq:u_e-est}
    \begin{split}&\left(\dashint_{B_{\ka r}^+}\left(|\D_{x'}\hat u-\mean{\D_{x'}\hat v}_{B_{\ka r}^+}|^p+|\hat U|^p\right)d\mu\right)^{1/p}\\
    &\le C\left(\dashint_{B_{\ka r}^+}\left(|\D_{x'}\hat v-\mean{\D_{x'}\hat v}_{B_{\ka r}^+}|^p+|\hat V|^p\right)d\mu\right)^{1/p}+C\left(\dashint_{B_{\ka r}^+}|\D \hat w|^pd\mu\right)^{1/p}\\
    &\le C\ka \left(\dashint_{B_r^+}\left(|\D_{x'}\hat v-\bq'|^p+|\hat V|^p\right)d\mu\right)^{1/p}+C\left(\dashint_{B_{\ka r}^+}|\D \hat w|^pd\mu\right)^{1/p}\\
    &\le C\ka \left(\dashint_{B_r^+}\left(|\D_{x'}\hat u-\bq'|^p+|\hat U|^p\right)d\mu\right)^{1/p}+C\ka^{-\frac{n+\al}p}\left(\dashint_{B_{r}^+}|\D \hat w|^pd\mu\right)^{1/p}.
\end{split}\end{align}
From $u=\hat u-\bar{a}^{-1}_{nn}\bar g_nx_n$, we find $\D_{x'}u=\D_{x'}\hat u$ and 
\begin{align*}
    U=\mean{\bar A\D u+\bar\bg,\vec{e}_n}+\mean{(A-\bar A)\D u+\bg-\bar\bg,\vec{e}_n}=\hat U+\mean{(A-\bar A)\D u+\bg-\bar\bg,\vec{e}_n}.
\end{align*}
Combining this with \eqref{eq:u_e-est} and \eqref{eq:deg-sol-hom-diff-est} produces
\begin{align*}
    &\left(\dashint_{B_{\ka r}^+}\left(|\D_{x'}u-\mean{\D_{x'}\hat v}_{B_{\ka r}^+}|^p+|U|^p\right)d\mu\right)^{1/p}\\
    &\le C\left(\dashint_{B_{\ka r}^+}\left(|\D_{x'}\hat u-\mean{\D_{x'}\hat v}_{B_{\ka r}^+}|^p+|\hat U|^p\right)d\mu\right)^{1/p}\\
    &\qquad+C\left(\dashint_{B_{\ka r}^+}|\mean{(A-\bar A)\D u+\bg-\bar\bg,\vec{e}_n}|^pd\mu\right)^{1/p} \\
    &\le C\ka\left(\dashint_{B_{r}^+}\left(|\D_{x'}\hat u-\bq'|^p+|\hat U|^p\right)d\mu\right)^{1/p}+C\ka^{-\frac{n+\al}p}\left(\dashint_{B_r^+}|\D \hat w|^pd\mu\right)^{1/p}  \\
    &\qquad+C\left(\dashint_{B_{\ka r}^+}|\mean{(A-\bar A)\D u+\bg-\bar\bg,\vec{e}_n}|^pd\mu\right)^{1/p}\\
    &\le C\ka\left(\dashint_{B_{r}^+}\left(|\D_{x'}u-\bq'|^p+|U|^p\right)d\mu\right)^{1/p}+C\ka^{-\frac{n+\al}p}\left(\dashint_{B_r^+}|\D \hat w|^pd\mu\right)^{1/p}  \\
    &\qquad+C\ka^{-\frac{n+\al}p}\left(\dashint_{B_{2r}^+}|\mean{(A-\bar A)\D u+\bg-\bar\bg,\vec{e}_n}|^pd\mu\right)^{1/p}\\
    &\le C\ka \left(\dashint_{B_{r}^+}\left(|\D_{x'}u-\bq'|^p+|U|^p\right)d\mu\right)^{1/p}\\
    &\qquad+C\ka^{-\frac{n+\al}p}\left(\|\D u\|_{L^\infty(B^+_{2r})}\eta_A^{1,\mu}(2r)+\eta_\bg^{1,\mu}(2r)\right).
\end{align*}
Since $\bq'\in \R^{n-1}$ is arbitrary, we get
$$
\psi(0,\ka r)\le C\ka\psi(0,r)+C\ka^{-\frac{n+\al}p}\left(\|\D u\|_{L^\infty(B^+_{2r})}\eta_A^{1,\mu}(2r)+\eta_\bg^{1,\mu}(2r)\right).
$$
By taking $\ka$ small so that $C\ka\le \ka^{\be}$ and using a standard iteration, we obtain \eqref{eq:deg-phi-est-bdry}.   
\end{proof}

Next, we consider the interior case when the ball $B_r(x_0)$ is away from the boundary $B'_4$. It should be noted that when $r<(x_0)_n/2$, $u$ solves an uniformly elliptic equation
$$
-\div(A^{x_0}\D u)=\div(\bg^{x_0})\quad\text{in }B_r(x_0),
$$
where $A^{x_0}:=(x_n/(x_0)_n)^\al A$ and $\bg^{x_0}:=(x_n/(x_0)_n)^\al\bg$. It may seem tempting to directly apply the result from \cite{DonKim17} to obtain estimates involving $\vp$. However, this approach leads to the mean-oscillations of $A^{x_0}$ and $\bg^{x_0}$. When bounding them by those of $A$ and $\bg$, we cannot avoid extra terms that depend on $(x_0)_n$. Handling these terms becomes challenging as the center $x_0$ approaches to the boundary $B'_4$.

To rectify this issue, we utilize the concept of partially Dini mean oscillation (partially DMO). A function $f\in L^1(B_1^+,dx)$ is said to be of \emph{$L^1(dx)$-partially DMO with respect to $x'$} in $B_1^+$ if
$$
\eta_f^{x'}(r):=\sup_{x_0\in B_1^+}\dashint_{B_r(x_0)\cap B_1^+}\left|f(x)-\dashint_{B'_r(x'_0)\cap B'_1}f(y',x_n)dy'\right|dx, \quad 0<r<1,
$$
is Dini. This partially DMO requirement is weaker than the standard DMO, and it holds that $\eta^{x'}_f(r)\le C\eta_f(r)$. For further understanding of the regularity results under the partially DMO condition, one can refer to \cites{Don12, DonXu19}.

Now we instead view new data $A^{x_0}$ and $\bg^{x_0}$ as being of partially DMO with respect to $x'$-variable, and apply a result in \cite{DonXu19}. As we will see in the proof of Lemma~\ref{lem:deg-int}, this approach allows us to avoid the additional term containing $(x_0)_n$, as $(x_n/(x_0)_n)^\al$ remains constant with respect to $x'$.

When $x_0=(x_0',(x_0)_n)\in B_3^+$ is away from the thin ball $B_4'$ and there is no confusion, we write for simplicity
$$
d:=(x_0)_n=\dist(x_0,B_4')>0.
$$

\begin{lemma}
    \label{lem:deg-int}
Let $x_0\in B_3^+$ and $0<\be<1$. For any $0<\rho<r\le d/2$,
\begin{align}
    \label{eq:deg-int-est}
    \psi(x_0,\rho)\le C(\rho/r)^{\be}\psi(x_0,r)+C\|\D u\|_{L^\infty(B_r(x_0))}\tilde\eta_A^{1,\mu}(\rho)+C\tilde\eta_\bg^{1,\mu}(\rho),
\end{align}
where $C=C(n,\la,\al,p,\be)>0$ are constants and $\tilde\eta_\bullet^{1,\mu}$ is as in \eqref{eq:tilde-eta}.
\end{lemma}

\begin{proof}
As mentioned above, $u$ solves
$$
-\div(A^{x_0}\D u)=\div(\bg^{x_0})\quad\text{in }B_r(x_0),
$$
where $A^{x_0}=(x_n/d)^\al A$ and $\bg^{x_0}=(x_n/d)^\al\bg$ are partially DMO with respect to $x'$. Then, it can be deduced from the proof of \cite{DonXu19}*{Lemma~3.3} (see also \cite{ChoKimLee20}*{Lemma~2.5}) that for any $0<\ka<1/2$ and $0<r\le d/2$,
\begin{align*}
    \psi^0(x_0,\ka r)&\le C\ka\psi^0(x_0,r)\\
    &\,\,+C\ka^{-n/p}\Bigg(\|\D u\|_{L^\infty(B_r(x_0))}\dashint_{B_r(x_0)}\left|A^{x_0}(x)-\dashint_{B_r'(x_0')}A^{x_0}(y',x_n)dy'\right|dx\\
    &\qquad\qquad\qquad+\dashint_{B_r(x_0)}\left|\bg^{x_0}(x)-\dashint_{B_r'(x_0')}\bg^{x_0}(y',x_n)dy'\right|dx\Bigg),
\end{align*}
where $C=C(n,\al,\la,p)>0$ and
$$\psi^0(x_0,r)=\inf_{\bq\in\R^n}\left(\dashint_{B_r(x_0)}|\mean{\D_{x'}u,\mean{A^{x_0}\D u+\bg^{x_0},\vec{e}_n}}-\bq|^pdx\right)^{1/p}.$$
From the identity $\mean{A^{x_0}\D u+\bg^{x_0},\vec{e}_n}=U^{x_0}$ and the fact that $d/2<x_n<3d/2$ for every $x\in B_r(x_0)$, we have
$$
\psi(x_0,\rho)\le C\psi^0(x_0,\rho),\quad \psi^0(x_0,r)\le C\psi(x_0,r).
$$
Moreover,
\begin{align*}
    &\dashint_{B_r(x_0)}\left|A^{x_0}(x)-\dashint_{B_r'(x_0')}A^{x_0}(y',x_n)dy'\right|dx\\
    &\qquad\le C\dashint_{B_r(x_0)}\left|A(x)-\dashint_{B_r'(x_0')}A(y',x_n)dy'\right|d\mu(x)\\&
    \qquad\le C\eta_A^{1,\mu}(r).
\end{align*}
Clearly, a similar estimate holds for $\bg$. Thus, we obtain
$$
\psi(x_0,\ka r)\le C\ka\psi(x_0,r)+C\ka^{-n/p}\left(\|\D u\|_{L^\infty(B_r(x_0))}\eta_A^{1,\mu}(r)+\eta_\bg^{1,\mu}(r)\right).
$$
Taking $\ka\in(0,1/2)$ small and performing a standard iteration as before, we conclude \eqref{eq:deg-int-est}.
\end{proof}

We now combine the previous two lemmas to establish the following uniform decay estimate of $\psi$.

\begin{lemma}
    \label{lem:psi-decay}
    Suppose $0<\be<1$. Then, for any $x_0\in B_3^+$ and $0<\rho\le r<\frac1{14}$, we have
\begin{align}
    \label{eq:psi-decay-est}
    \psi(x_0,\rho)\le \begin{multlined}[t]C(\rho/r)^{\be}\left(\dashint_{B_{8r}(x_0)\cap B_4^+}|\D u|d\mu+\|\bg\|_\infty\right)\\+C\|\D u\|_{L^\infty(B_{14r}(x_0)\cap B_4^+)}\hat\eta_A^{1,\mu}(\rho)+C\hat\eta_\bg^{1,\mu}(\rho),
\end{multlined}\end{align}
where $C=C(n,\al,\la,p,\be)>0$ and $\hat\eta_\bullet^{1,\mu}$ is defined as in \eqref{eq:hat-eta}.
\end{lemma}

\begin{proof}
We divide the proof into the three cases:
$$
\rho< r<d/2,\quad d/2<\rho< r,\,\text{ or }\,\rho\le d/2\le r.
$$

\medskip\noindent\emph{Case 1.} We first consider the case $\rho<r<d/2$. By Lemma~\ref{lem:deg-int}
\begin{align}
    \label{eq:psi-decay-int}
    \psi(x_0,\rho)\le C(\rho/r)^{\be}\psi(x_0,r)+C\left(\|\D u\|_{L^\infty(B_r(x_0))}\tilde\eta_A^{1,\mu}(\rho)+\tilde\eta_\bg^{1,\mu}(\rho)\right).
\end{align}
Using $|U^{x_0}|\le C|A\D u+\bg|$ in $B_r(x_0)$, we get
\begin{align}\label{eq:psi-bound}
\begin{split}\psi(x_0,r)&\le\left(\dashint_{B_r(x_0)}|\mean{\D_{x'}u,U^{x_0}}|^pd\mu\right)^{1/p}\le C\left(\dashint_{B_r(x_0)}|\D u|d\mu+\|\bg\|_\infty\right).\end{split}
\end{align}
This, along with \eqref{eq:psi-decay-int} and the doubling property of $\mu$ (Lemma~\ref{lem:doubl}), implies \eqref{eq:psi-decay-est}.

\medskip\noindent\emph{Case 2.} Suppose $d/2<\rho<r$. For $\bar x_0:=(x_0',0)\in B'_3$, the doubling and Lemma~\ref{lem:deg-bdry} yield
\begin{align*}
    \psi(x_0,\rho)&\le C\psi(\bar x_0,3\rho)\\
    &\le C(\rho/r)^{\be}\psi(\bar x_0,3r)+C\left(\|\D u\|_{L^\infty(B_{6r}^+(\bar x_0))}\tilde\eta_A^{1,\mu}(6\rho)+\tilde\eta_\bg^{1,\mu}(6\rho)\right).
\end{align*}
Moreover, we have
\begin{align}\label{eq:psi-bound-1}\begin{split}
    \psi(\bar x_0,3r)&\le \left(\dashint_{B_{3r}^+(\bar x_0)}|\mean{\D_{x'}u,U}|^pd\mu\right)^{1/p}\le C\left(\dashint_{B_{5r}(x_0)\cap B_4^+}|\D u|d\mu+\|\bg\|_\infty\right).
\end{split}\end{align}
Thus,
\begin{align*}
\psi(x_0,\rho)&\le C(\rho/r)^{\be}\left(\dashint_{B_{5r}(x_0)\cap B_4^+}|\D u|d\mu+\|\bg\|_\infty\right)\\
&\quad +C\left(\|\D u\|_{L^\infty(B_{8r}(x_0)\cap B_4^+)}\tilde\eta_A^{1,\mu}(6\rho)+\tilde\eta_\bg^{1,\mu}(6\rho)\right).
\end{align*}
From the fact that $\tilde\eta_\bullet^{1,\mu}\le \hat\eta_\bullet^{1,\mu}$ and $t\longmapsto \frac{\hat\eta_\bullet^{1,\mu}(t)}{t^{\be}}$ is nonincreasing, we infer
$$
\tilde\eta_\bullet^{1,\mu}(6\rho)\le \hat\eta_\bullet^{1,\mu}(6\rho)\le 6^{\be}\hat\eta_\bullet^{1,\mu}(\rho).
$$
This concludes \eqref{eq:psi-decay-est}.

\medskip\noindent\emph{Case 3.} It remains to deal with the case $\rho\le d/2\le r$. We observe that
\begin{align*}
    \left(\dashint_{B_{d/2}(x_0)}|U^{x_0}|^pd\mu\right)^{1/p}&\le C\left(\dashint_{B_{d/2}(x_0)}|U|^pd\mu\right)^{1/p}\le C\left(\dashint_{B_{3d/2}(\bar x_0)}|U|^pd\mu\right)^{1/p}\\
    &\le C\psi(\bar x_0,3d/2),
\end{align*}
which gives
\begin{align*}
    \psi(x_0,d/2)&\le C\inf_{\bq'\in\R^{n-1}}\left(\dashint_{B_{d/2}(x_0)}|\D_{x'}u-\bq'|^pd\mu\right)^{1/p}\\
    &\qquad +C\left(\dashint_{B_{d/2}(x_0)}|U^{x_0}|^pd\mu\right)^{1/p}\\
    &\le C\psi(\bar x_0,3d/2).
\end{align*}
By using this and Lemmas~\ref{lem:deg-bdry} and \ref{lem:deg-int}, we derive
\begin{align*}
    \psi(x_0,\rho)&\le C(\rho/d)^\be\psi(x_0,d/2)+C\left(\|\D u\|_{L^\infty(B_{d/2}(x_0))}\tilde\eta_A^{1,\mu}(\rho)+\tilde\eta_\bg^{1,\mu}(\rho)\right)\\
    &\le C(\rho/d)^\be\psi(\bar x_0,3d/2)+C\left(\|\D u\|_{L^\infty(B_{d/2}(x_0))}\tilde\eta_A^{1,\mu}(\rho)+\tilde\eta_\bg^{1,\mu}(\rho)\right)\\
    &\le C(\rho/r)^\be\psi(\bar x_0,3r)+C\left(\|\D u\|_{L^\infty(B_{d/2}(x_0))}\tilde\eta_A^{1,\mu}(\rho)+\tilde\eta_\bg^{1,\mu}(\rho)\right)\\
    &\qquad+C(\rho/d)^\be\left(\|\D u\|_{L^\infty(B_{6r}^+(\bar x_0))}\tilde\eta_A^{1,\mu}(3d)+\tilde\eta_\bg^{1,\mu}(3d)\right)\\
    &\le C(\rho/r)^\be\left(\dashint_{B_{5r}(x_0)\cap B_4^+}|\D u|d\mu+\|\bg\|_\infty\right)\\
    &\qquad+C\left(\|\D u\|_{L^\infty(B_{8r}(x_0)\cap B^+_4)}\hat\eta_A^{1,\mu}(\rho)+\hat\eta_\bg^{1,\mu}(\rho)\right),
\end{align*}
where in the last step we used $\tilde\eta_\bullet^{1,\mu}\le \hat\eta_\bullet^{1,\mu}$ and the monotonicity of $s\longmapsto \frac{\hat\eta_\bullet^{1,\mu}(s)}{s^{\be}}$. This completes the proof.
\end{proof}
Our next objective is the $L^\infty$-estimate of $\D u$, which enables us to remove the $C^1$ assumption on $u$ by a standard approximation argument. For this purpose, we follow the idea in \cite{DonEscKim18}*{Lemma~2.11}.
For $x_0\in B_3^+$ and $0<r<1/14$, we take a vector $\bq_{x_0,r}\in\R^n$ such that
\begin{align}\label{eq:psi-inf}
    \psi(x_0,r)=\begin{cases}
        \left(\dashint_{B_r(x_0)}|\mean{\D_{x'}u,U^{x_0}}-\bq_{x_0,r}|^pd\mu\right)^{1/p},&0<r\le d/2,\\
        \left(\dashint_{B_r(x_0)\cap B_4^+}|\mean{\D_{x'}u, U}-\bq_{x_0,r}|^pd\mu\right)^{1/p},&d/2<r<1/14.
    \end{cases}
\end{align}
Note that the last component of $\bq_{x_0,r}$ is zero when $d/2<r<1/14$.

\begin{lemma}\label{lem:gradient-unif-est}
It holds that
    \begin{align}\label{eq:gradient-unif-est}
        \|\D u\|_{L^\infty(B_2^+)}\le C\int_{B_4^+}|\D u|d\mu+C\int_0^1\frac{\hat\eta_\bg^{1,\mu}(t)}tdt+C\|\bg\|_{L^\infty(B_4^+)},
    \end{align}
    for some constant $C>0$ depending only on $n,\la,\al,p$.
\end{lemma}

\begin{proof}
We split the proof of this lemma into two steps.

\medskip\noindent\emph{Step 1.}
We claim that for any $x_0\in B_3^+$ and $0<r\le 1/18$,
\begin{align}
    \label{eq:gradient-infty-bound}
    |\D u(x_0)|\le \begin{multlined}[t]C\dashint_{B_{10r}(x_0)\cap B_4^+}|\D u|d\mu+C\|\bg\|_\infty\\
    +C\|\D u\|_{L^\infty(B_{18r}(x_0)\cap B_4^+)}\int_0^r\frac{\hat\eta_A^{1,\mu}(t)}tdt+C\int_0^r\frac{\hat\eta_\bg^{1,\mu}(t)}tdt.
\end{multlined}\end{align}
To prove the claim \eqref{eq:gradient-infty-bound}, we consider two cases:
$$
0<r\le d/2 \,\,\,\text{ or }\,\,\, d/2<r\le 1/18.
$$

\medskip\noindent\emph{Case 1.} Suppose $r\le d/2$. We take average of the trivial inequality
$$
|\bq_{x_0,r}-\bq_{x_0,r/2}|^p\le |\mean{\D_{x'}u(x),U^{x_0}(x)}-\bq_{x_0,r}|^p+|\mean{\D_{x'}u(x),U^{x_0}(x)}-\bq_{x_0,r/2}|^p
$$
over $x\in B_{r/2}(x_0)$ with respect to $d\mu$ and take the $p$-th root to get
$$
|\bq_{x_0,r}-\bq_{x_0,r/2}|\le C\left(\psi(x_0,r)+\psi(x_0,r/2)\right).
$$
By iterating, we further have
\begin{align}
    \label{eq:q-diff-est}
    |\bq_{x_0,2^{-k}r}-\bq_{x_0,r}|\le C\sum_{j=0}^k\psi(x_0,2^{-j}r).
\end{align}
Note that by \eqref{eq:psi-decay-est}
$$
\lim_{k\to +\infty}\psi(x_0,2^{-k}r)=0,
$$
which along with the assumption $u\in C^1(\overline{B_3^+})$ implies
$$
\lim_{k\to +\infty}\bq_{x_0,2^{-k}r}=\mean{\D_{x'}u(x_0),U^{x_0}(x_0)}=\mean{\D_{x'}u(x_0),U(x_0)}.
$$
Thus, by taking $k\to +\infty$ in \eqref{eq:q-diff-est} and using \eqref{eq:psi-decay-est}, we obtain
\begin{align}
    \label{eq:u-q-diff-est}
    \begin{split}
        &|\mean{\D_{x'}u(x_0),U(x_0)}-\bq_{x_0,r}|\\
        &\qquad\le \begin{multlined}[t]C\dashint_{B_{8r}(x_0)\cap B_4^+}|\D u|d\mu+C\|\bg\|_\infty
        \\+C\|\D u\|_{L^\infty(B_{14r}(x_0)\cap B_4^+)}\int_0^r\frac{\hat\eta_A^{1,\mu}(t)}tdt+C\int_0^r\frac{\hat\eta_\bg^{1,\mu}(t)}tdt.
        \end{multlined}
    \end{split}
\end{align}
On the other hand, we have for any $x\in B_r(x_0)$,
\begin{align*}
    |\bq_{x_0,r}|^p&\le |\mean{\D_{x'}u(x),U^{x_0}(x)}-\bq_{x_0,r}|^p+|\mean{\D_{x'}u(x),U^{x_0}(x)}|^p\\
    &\le |\mean{\D_{x'}u(x),U^{x_0}(x)}-\bq_{x_0,r}|^p+C\left(|\D u(x)|^p+|\bg(x)|^p\right).
\end{align*}
Taking average of this over $x\in B_r(x_0)$ with respect to $d\mu$ and taking the $p$-th root yield
$$
|\bq_{x_0,r}|\le C\psi(x_0,r)+C\left(\dashint_{B_r(x_0)}|\D u|^pd\mu\right)^{1/p}+C\|\bg\|_\infty.
$$
Due to \eqref{eq:psi-bound}, we further have
$$
|\bq_{x_0,r}|\le C\dashint_{B_r(x_0)}|\D u|d\mu+C\|\bg\|_\infty.
$$
Combining this with \eqref{eq:u-q-diff-est}, we infer
\begin{align*}
    |\D u(x_0)|&\le C\left(|\mean{\D_{x'}u(x_0),U(x_0)}|+\|\bg\|_\infty\right)\\
    &\le \begin{multlined}[t] C\dashint_{B_{8r}(x_0)\cap B_4^+}|\D u|d\mu+C\|\bg\|_\infty\\
    +C\|\D u\|_{L^\infty(B_{14r}(x_0)\cap B_4^+)}\int_0^r\frac{\hat\eta_A^{1,\mu}(t)}tdt+C\int_0^r\frac{\hat\eta_\bg^{1,\mu}(t)}tdt.
    \end{multlined}
\end{align*}

\medskip\noindent\emph{Case 2.} Next, we consider the case $d/2<r<1/18$. We take a nonnegative integer $j_0$ such that $2^{-(j_0+1)}r\le d/2<2^{-j_0}r$. By using the idea at the beginning of Case 1, we can easily obtain that for every $j\ge0$ with $j\neq j_0$
$$
|\bq_{x_0,2^{-j}r}-\bq_{x_0,2^{-(j+1)}r}|\le C\left(\psi(x_0,2^{-j}r)+\psi(x_0,2^{-(j+1)}r)\right).
$$
However, the bound is nontrivial when $j=j_0$ due to the discrepancy between $U^{x_0}$ and $U$. By iteration, we have that for any $k\ge j_0+1$
\begin{align}\begin{split}
    \label{eq:q-diff-est-sum}
    |\bq_{x_0,r}-\bq_{x_0,2^{-k}r}|\le C\sum_{j=0}^{+\infty}\psi(x_0,2^{-j}r)+|\bq_{x_0,2^{-j_0}r}-\bq_{x_0,2^{-(j_0+1)}r}|.
\end{split}\end{align}
We can use \eqref{eq:psi-decay-est} to estimate the first term in the right-hand side of \eqref{eq:q-diff-est-sum}:
\begin{align}
    \label{eq:psi-est-sum}
    \sum_{j=0}^{+\infty}\psi(x_0,2^{-j}r)\le \begin{multlined}[t] C\dashint_{B_{8r}(x_0)\cap B_4^+}|\D u|d\mu+C\|\bg\|_\infty\\
    +C\|\D u\|_{L^\infty(B_{14r}(x_0)\cap B_4^+)}\int_0^r\frac{\hat\eta_A^{1,\mu}(t)}tdt+C\int_0^r\frac{\hat\eta_\bg^{1,\mu}(t)}tdt.
    \end{multlined}
\end{align}
To treat the second term, we observe that for $x\in B_{2^{-(j_0+1)}r}(x_0)$
\begin{multline*}
    |\bq_{x_0,2^{-j_0}r}-\bq_{x_0,2^{-(j_0+1)}r}|^p\\
    \le |\mean{\D_{x'}u,U^{x_0}}-\bq_{x_0,2^{-(j_0+1)}r}|^p+|\mean{\D_{x'}u,U}-\bq_{x_0,2^{-j_0}r}|^p+|U^{x_0}-U|^p.
\end{multline*}
Arguing as before, we can deduce from this inequality
\begin{multline*}
    |\bq_{x_0,2^{-j_0}r}-\bq_{x_0,2^{-(j_0+1)}r}|\\
    \le C\psi(x_0,2^{-(j_0+1)}r)+C\psi(x_0,2^{-j_0}r)+C\left(\dashint_{B_{2^{-(j_0+1)}r}(x_0)}|U^{x_0}-U|^pd\mu\right)^{1/p}.
\end{multline*}
We recall $2^{-(j_0+1)}r\le d/2< 2^{-j_0}r$ and apply Lemma~\ref{lem:deg-bdry} to obtain
\begin{align*}
    &\left(\dashint_{B_{2^{-(j_0+1)}r}(x_0)}|U^{x_0}-U|^p d\mu\right)^{1/p}\\
    &\qquad\le C\left(\dashint_{B_{d/2}(x_0)}|U|^pd\mu\right)^{1/p}\le C\left(\dashint_{B_{2^{2-j_0}r}(\bar x_0)}|U|^pd\mu\right)^{1/p}\\
    &\qquad\le C\psi(\bar x_0,2^{3-j_0}r)+C\|\D u\|_{L^\infty(B^+_{16r}(\bar x_0))}\eta_A^{1,\mu}(2^{4-j_0}r)+C\eta_\bg^{1,\mu}(2^{4-j_0}r)\\
    &\qquad\le C\psi(\bar x_0,8r)+C\|\D u\|_{L^\infty(B^+_{16r}(\bar x_0))}\tilde\eta_A^{1,\mu}(2^{4-j_0}r)+C\tilde\eta_\bg^{1,\mu}(2^{4-j_0}r).
\end{align*}
By combining the preceding two estimates, we get
\begin{align}\label{eq:q-diff-est-1}\begin{split}
    &|\bq_{x_0,2^{-j_0}r}-\bq_{x_0,2^{-(j_0+1)}r}|\\
    &\qquad\le \begin{multlined}[t]C\psi(x_0,2^{-(j_0+1)}r)+C\psi(x_0,2^{-j_0}r)+C\psi(\bar x_0,8r)\\
    +C\|\D u\|_{L^\infty(B_{18r}(x_0)\cap B_4^+)}\tilde\eta_A^{1,\mu}(2^{4-j_0}r)+C\tilde\eta_\bg^{1,\mu}(2^{4-j_0}r).
    \end{multlined}\end{split}
\end{align}
We can argue as in \eqref{eq:psi-bound-1} to get
$$
\psi(\bar x_0,8r)\le C\dashint_{B_{10r}(x_0)\cap B_4^+}|\D u|d\mu+\|\bg\|_\infty,
$$
which combined with \eqref{eq:q-diff-est-1} implies
\begin{align}\label{eq:q-diff-est-j}\begin{split}
    &|\bq_{x_0,2^{-j_0}r}-\bq_{x_0,2^{-(j_0+1)}r}|\\
    &\le \begin{multlined}[t]C\psi(x_0,2^{-(j_0+1)}r)+C\psi(x_0,2^{-j_0}r)+C\dashint_{B_{10r}(x_0)\cap B_4^+}|\D u|d\mu\\
    +C\|\D u\|_{L^\infty(B_{18r}(x_0)\cap B_4^+)}\tilde\eta_A^{1,\mu}(2^{4-j_0}r)+C\tilde\eta_\bg^{1,\mu}(2^{4-j_0}r)+\|\bg\|_\infty.
    \end{multlined}\end{split}
\end{align}
In addition, by using that $\hat\eta_\bullet^{1,\mu}$ satisfies \eqref{eq:alm-mon}, one can easily show that
$$
\tilde\eta_\bullet^{1,\mu}(2^{4-j_0}r)\le C\int_0^{r}\frac{\hat\eta_\bullet^{1,\mu}(t)}tdt.
$$
Combining this with \eqref{eq:q-diff-est-sum}, \eqref{eq:psi-est-sum} and \eqref{eq:q-diff-est-j} and taking $k\to \infty$ yield
\begin{align*}
    &|\mean{\D_{x'}u(x_0),U(x_0)}-\bq_{x_0,r}|\\
    &\qquad\le \begin{multlined}[t]C\dashint_{B_{10r}(x_0)\cap B_4^+}|\D u|d\mu+C\|\bg\|_\infty\\
    + C\|\D u\|_{L^\infty(B_{18r}(x_0)\cap B_4^+)}\int_0^r\frac{\hat\eta_A^{1,\mu}(t)}tdt+C\int_0^r\frac{\hat\eta_\bg^{1,\mu}(t)}tdt.\end{multlined}
\end{align*}
On the other hand, we can obtain the following estimate by arguing as in Case 1
$$
|\bq_{x_0,r}|\le C\dashint_{B_r(x_0)\cap B_4^+}|\D u|d\mu+C\|\bg\|_\infty.
$$
The previous two estimates imply \eqref{eq:gradient-infty-bound}.

\medskip\noindent\emph{Step 2.} We are now ready to prove \eqref{eq:gradient-unif-est}. For $k\in \mathbb{N}$, we denote $s_k:=3-2^{1-k}$, so that $s_{k+1}-s_k=2^{-k}$, $s_1=2$ and $s_k\nearrow 3$. We note that for every $x_0\in B_{s_k}^+$ and $r= 2^{-k-5}$, $B_{18r}(x_0)\cap B_4^+\subset B_{s_{k+1}}^+$. For $C>0$ as in \eqref{eq:gradient-infty-bound} and $\al^+=\max\{\al,0\}$, we fix $0<r_0<1/4$ small so that
$$
C\int_0^{r_0}\frac{\hat\eta_{A}^{1,\mu}(t)}tdt<3^{-(n+\al^+)},
$$
and take $k_0\in \mathbb{N}$ such that $2^{-k_0-5}<r_0$. It is easily seen that $\mu(B_r(x_0)\cap B_4^+)\ge c(n,\al)r^{n+\al^+}$ whenever $x_0\in B_3^+$ and $0<r<1$. Then, we have by \eqref{eq:gradient-infty-bound} that for every $k\ge k_0$
\begin{align*}
    \|\D u\|_{L^\infty(B_{s_k}^+)}&\le C2^{k(n+\al^+)}\int_{B_4^+}|\D u|d\mu+C\|\bg\|_\infty+3^{-(n+\al^+)}\|\D u\|_{L^\infty(B_{s_{k+1}}^+)}\\
    &\quad +C\int_0^1\frac{\hat\eta_\bg^{1,\mu}(t)}tdt.
\end{align*}
We multiply this by $3^{-k(n+\al^+)}$ and take summation over $k\ge k_0$ to get
\begin{align*}
    \sum_{k=k_0}^{+\infty} 3^{-k(n+\al^+)}\|\D u\|_{L^\infty(B^+_{s_k})}&\le 
    C\int_{B_4^+}|\D u|d\mu+C\int_0^1\frac{\hat\eta_\bg^{1,\mu}(t)}tdt+C\|\bg\|_{L^\infty(B_4^+)}\\
    &\quad +\sum_{k=k_0}^{+\infty} 3^{-(k+1)(n+\al^+)}\|\D u\|_{L^\infty(B_{s_{k+1}}^+)}.
\end{align*}
Due to our assumption $u\in C^1(\overline{B_3^+})$, the sum $\sum_{k=k_0}^{+\infty} 3^{-k(n+\al^+)}\|\D u\|_{L^\infty(B_{s_k}^+)}$ converges, and hence \eqref{eq:gradient-unif-est} follows.
\end{proof}

To proceed, given $0<\be<1$, we consider a modulus of continuity $\omega:[0,1)\to[0,+\infty)$ defined by
\begin{align}
    \label{eq:mod-of-conti}
    \begin{split}
        \omega(r)
        &=\left(\int_{B_4^+}|\D u|d\mu+\|\bg\|_{L^\infty(B_4^+)}\right) r^{\be}+\int_0^r\frac{\hat\eta_\bg^{1,\mu}(t)}tdt\\
        &\qquad+\left(\int_{B_4^+}|\D u|d\mu+\|\bg\|_{L^\infty(B_4^+)}+\int_0^1\frac{\hat\eta_\bg^{1,\mu}(t)}tdt \right)\int_0^r\frac{\hat\eta_A^{1,\mu}(t)}tdt.
    \end{split}
\end{align}

\begin{lemma}
    For any $x_0\in B_1^+$ and $0<r<1/18$, we have
    \begin{align}\label{eq:u-q-diff}\begin{split}
        &|\mean{\D_{x'}u(x_0),U(x_0)}-\bq_{x_0,r}|\le C\omega(r)
    \end{split}\end{align}
for some constant $C>0$ depending only on $n,\la,\al,p,\be$.
\end{lemma}

\begin{proof}
Recall the identity $\lim_{k\to +\infty}\bq_{x_0,2^{-k}r}=\mean{\D_{x'}u(x_0),U(x_0)}$, and, as before, consider two cases: either $0<r\le d/2$ or $d/2<r<1/18$.

\medskip\noindent\emph{Case 1.} If $0<r\le d/2$, then we have
\begin{align*}
    |\mean{\D_{x'}u(x_0),U(x_0)}-\bq_{x_0,r}|\le \sum_{j=0}^{+\infty}|\bq_{x_0,2^{-j}r}-\bq_{x_0,2^{-(j+1)}r}|\le C\sum_{j=0}^{+\infty}\psi(x_0,2^{-j}r).
\end{align*}
By using \eqref{eq:psi-decay-est}, we obtain
\begin{align}
    \label{eq:psi-sum-est}
        \sum_{j=0}^{+\infty}\psi(x_0,2^{-j}r)
        \le \begin{multlined}[t]C\left(\int_{B_4^+}|\D u|d\mu+\|\bg\|_{L^\infty(B_4^+)}\right)r^{\be}\\
        +C\|\D u\|_{L^\infty(B_2^+)}\int_0^r\frac{\hat\eta_A^{1,\mu}(t)}tdt+C\int_0^r\frac{\hat\eta_\bg^{1,\mu}(t)}tdt,\end{multlined}
\end{align}
and hence \eqref{eq:u-q-diff} follows from Lemma~\ref{lem:gradient-unif-est}.

\medskip\noindent\emph{Case 2.} Suppose $d/2<r<1/18$. Take $j_0\ge0$ such that $2^{-(j_0+1)}r\le d/2<2^{-j_0}r$. By first sending $k\to +\infty$ in \eqref{eq:q-diff-est-sum} and then applying \eqref{eq:q-diff-est-1} and \eqref{eq:deg-phi-est-bdry} sequentially, we get
\begin{align*}
    &|\mean{\D_{x'}u(x_0),U(x_0)}-\bq_{x_0,r}|\\
    &\le C\sum_{j=0}^{+\infty}\psi(x_0,2^{-j}r)+|\bq_{x_0,2^{-j_0}r}-\bq_{x_0,2^{-(j_0+1)}r}|\\
    &\le C\sum_{j=0}^{+\infty}\psi(x_0,2^{-j}r)+C\vp(\bar x_0,8r)+C\|\D u\|_{L^\infty(B_2^+)}\tilde\eta_A^{1,\mu}(2^{4-j_0}r)+C\tilde\eta_\bg^{1,\mu}(2^{4-j_0}r)\\
    &\le C\sum_{j=0}^{+\infty}\psi(x_0,2^{-j}r)+Cr^{\be}\int_{B_4^+}|\D u|d\mu+C\left(\tilde\eta_\bg^{1,\mu}(16r)+\tilde\eta_\bg^{1,\mu}(2^{4-j_0}r)\right)\\
    &\quad +C\|\D u\|_{L^\infty(B_2^+)}\left(\tilde\eta_A^{1,\mu}(16r)+\tilde\eta_A^{1,\mu}(2^{4-j_0}r)\right).
\end{align*}
This, together with \eqref{eq:gradient-unif-est} and \eqref{eq:psi-sum-est}, implies \eqref{eq:u-q-diff}.
\end{proof}

We are now ready to prove Theorem~\ref{thm:deg-pde}.

\begin{proof}[Proof of Theorem~\ref{thm:deg-pde}]
Our goal is to show that for any $x_0,y_0\in B_1^+$ with $r:=|x_0-y_0|>0$,
\begin{align}
    \label{eq:gradient-mod-conti}
    |\D u(x_0)-\D u(y_0)|\le C\omega(r),
\end{align}
where $C>0$ is a constant, depending only on $n,\la,\al,p,\be$, and $\omega$ is a modulus of continuity as in \eqref{eq:mod-of-conti}.

If $r\ge 1/18$, then we can simply use $|\D u(x_0)-\D u(y_0)|\le 2\|\D u\|_{L^\infty(B_1^+)}\le 36\|\D u\|_{L^\infty(B_1^+)}r$ and apply \eqref{eq:gradient-unif-est} to get \eqref{eq:gradient-mod-conti}. Thus, we may assume $0<r<1/18$. We consider two cases either $r\ge (x_0)_n/8$ or $r<(x_0)_n/8$.

\medskip\noindent\emph{Case 1.} Suppose $r\ge (x_0)_n/8$. Since the monotonicity of $t\longmapsto \frac{\hat\eta_\bullet^{1,\mu}(t)}{t^{\be}}$ implies that of $t
\longmapsto \frac{\omega(t)}{t^{\be}}$, we have $\omega(5r)\le C\omega(r)$. This, along with \eqref{eq:u-q-diff}, gives
\begin{align}
    \label{eq:u-gradient-conti-est}
    \begin{split}&|\mean{\D_{x'}u(x_0),U(x_0)}-\mean{\D_{x'}u(y_0),U(y_0)}|\\
    &\qquad\le \begin{multlined}[t]|\mean{\D_{x'}u(x_0),U(x_0)}-\bq_{x_0,5r}|+|\mean{\D_{x'}u(y_0)-U(y_0)}-\bq_{y_0,5r}|\\+|\bq_{x_0,5r}-\bq_{y_0,5r}|\end{multlined}\\
    &\qquad\le C\omega(r)+|\bq_{x_0,5r}-\bq_{y_0,5r}|.
\end{split}\end{align}
To treat the last term $|\bq_{x_0,5r}-\bq_{y_0,5r}|$, we observe that the assumption $r\ge (x_0)_n/8$ gives $\frac{(y_0)_n}2\le\frac{(x_0)_n+r}2<5r$. We then take average of the inequality
\begin{multline*}
    |\bq_{x_0,5r}-\bq_{y_0,5r}|^p\le |\bq_{x_0,5r}-\mean{\D_{x'}u(x),U(x)}|^p+|\bq_{y_0,5r}-\mean{\D_{x'}u(x), U(x)}|^p
\end{multline*}
over $x\in B_r(x_0)\cap B_4^+$ and take the $p$-th root to obtain
\begin{align*}
    |\bq_{x_0,5r}-\bq_{y_0,5r}|
    &\le \begin{multlined}[t]C\left(\dashint_{B_{5r}(x_0)\cap B_4^+}|\mean{\D_{x'}u,U}-\bq_{x_0,5r}|^pd\mu\right)^{1/p}\\
    +C\left(\dashint_{B_{5r}(y_0)\cap B_4^+}|\mean{\D_{x'}u,U}-\bq_{y_0,5r}|^pd\mu\right)^{1/p}\end{multlined}\\
    &\le C\psi(x_0,5r)+C\psi(y_0,5r)\\
    &\le C\omega(r),
\end{align*}
where we applied \eqref{eq:psi-decay-est} in the last step. By combining this with \eqref{eq:u-gradient-conti-est} and using the definition of $U$, we get \eqref{eq:gradient-mod-conti}.

\medskip\noindent\emph{Case 2.} Now we suppose $r<(x_0)_n/8$. We argue as in \eqref{eq:u-gradient-conti-est} to get
\begin{align*}
    |\mean{\D_{x'}u(x_0),U(x_0)}-\mean{\D_{x'}u(y_0),U(y_0)}|\le C\omega(r)+|\bq_{x_0,2r}-\bq_{y_0,2r}|.
\end{align*}
To estimate the last term, we note that $(y_0)_n>(x_0)_n-r>7r$, which implies $2r<\min\{(x_0)_n/2,(y_0)_n/2\}$. We then use the trivial inequality
\begin{align*}
    |\bq_{x_0,2r}-\bq_{y_0,2r}|^p&\le |\bq_{x_0,2r}-\mean{\D_{x'}u(x), U^{x_0}(x)}|^p+|\bq_{y_0,2r}-\mean{\D_{x'}u(x), U^{y_0}(x)}|^p\\
    &\quad +|U^{x_0}(x)-U^{y_0}(x)|^p
\end{align*}
for $x\in B_r(x_0)$ to deduce
    \begin{align*}
    &|\bq_{x_0,2r}-\bq_{y_0,2r}|\\
    &\le C\left(\dashint_{B_{2r}(x_0)}|\mean{\D_{x'}u,U^{x_0}}-\bq_{x_0,2r}|^pd\mu\right)^{1/p}\\
    &\quad +C\left(\dashint_{B_{2r}(y_0)}|\mean{\D_{x'}u,U^{y_0}}-\bq_{y_0,2r}|^pd\mu\right)^{1/p}+C\left(\dashint_{B_r(x_0)}|U^{x_0}-U^{y_0}|^pd\mu\right)^{1/p}\\
    &\le C\psi(x_0,2r)+C\psi(y_0,2r)+C\left(\dashint_{B_r(x_0)}|U^{x_0}-U^{y_0}|^pd\mu\right)^{1/p}\\
    &\le C\omega(r)+C\left(\dashint_{B_r(x_0)}|U^{x_0}-U^{y_0}|^pd\mu\right)^{1/p}.
\end{align*}
Thus, it is sufficient to show that $\left(\dashint_{B_r(x_0)}|U^{x_0}-U^{y_0}|^pd\mu\right)^{1/p}\le C\omega(r)$. To this end, we observe that in $B_r(x_0)$,
$$
|U^{x_0}-U^{y_0}|=\left|1-\left(\frac{(x_0)_n}{(y_0)_n}\right)^\al\right| |U^{x_0}|\le \frac{Cr}{(x_0)_n}|U^{x_0}|,
$$
which gives
\begin{align}
    \label{eq:U-diff-est}
    \left(\dashint_{B_r(x_0)}|U^{x_0}-U^{y_0}|^pd\mu\right)^{1/p}\le \frac{Cr}{(x_0)_n}\left(\dashint_{B_r(x_0)}|U^{x_0}|^pd\mu\right)^{1/p}.
\end{align}
To estimate the right-hand side, we denote $d:=(x_0)_n$ and take $k\in\mathbb{N}$ such that $d/8<2^kr\le d/4$. For each $0\le j\le k-1$, notice that $2^{j+1}r\le {\color{blue}d}/4$ and let $\bq_{x_0,2^{j+1}r}=(\bq_{x_0,2^{j+1}r}',(\bq_{x_0,2^{j+1}r})_n)\in \R^n$ be as in \eqref{eq:psi-inf}. Then
\begin{align*}
    &\dashint_{B_{2^jr}(x_0)}|U^{x_0}|^pd\mu\\
    &=\dashint_{B_{2^jr}(x_0)}\dashint_{B_{2^{j+1}r}(x_0)}|U^{x_0}(y)+(U^{x_0}(x)-U^{x_0}(y))|^pd\mu(y)d\mu(x)\\
    &\le \dashint_{B_{2^jr}(x_0)}\dashint_{B_{2^{j+1}r}(x_0)}|U^{x_0}(y)|^pd\mu(y)d\mu(x)\\
    &\quad +C\dashint_{B_{2^{j+1}r}(x_0)}\dashint_{B_{2^{j+1}r}(x_0)}|U^{x_0}(x)-U^{x_0}(y)|^pd\mu(y)d\mu(x)\\
    &\le \dashint_{B_{2^{j+1}r}(x_0)}|U^{x_0}|^pd\mu+C\dashint_{B_{2^{j+1}r}(x_0)}|U^{x_0}-(\bq_{x_0,2^{j+1}r})_n|^pd\mu\\
    &\le \dashint_{B_{2^{j+1}r}(x_0)}|U^{x_0}|^pd\mu+C\left(\omega(2^{j+1}r)\right)^p,
\end{align*}
where we applied Lemma~\ref{lem:psi-decay} in the last step. By summing up the previous estimate over $0\le j\le k-1$, we attain
\begin{align*}
\dashint_{B_r(x_0)}|U^{x_0}|^pd\mu\le\dashint_{B_{2^kr}(x_0)}|U^{x_0}|^pd\mu+C\sum_{j=1}^k\left(\omega(2^jr)\right)^p,
\end{align*}
and hence by applying H\"older's inequality
$$
\left(\dashint_{B_r(x_0)}|U^{x_0}|^pd\mu\right)^{1/p}\le C\left(\dashint_{B_{2^kr}(x_0)}|U^{x_0}|^pd\mu\right)^{1/p}+Ck^{\frac{1-p}p}\sum_{j=1}^k\omega(2^jr).
$$
To bound the first term in the right-hand side, we denote $\bar x_0:=(x_0',0)\in B_1'$ and exploit Lemma~\ref{lem:deg-bdry} and the monotonicity of $t\longmapsto\frac{\omega(t)}{t^{\be}}$ to get
\begin{align*}
    &\left(\dashint_{B_{2^kr}(x_0)}|U^{x_0}|^pd\mu\right)^{1/p}\\
    &\qquad\le C\left(\dashint_{B_{d/4}(x_0)}|U^{x_0}|^pd\mu\right)^{1/p}\le C\left(\dashint_{B_{d/4}(x_0)}|\mean{A\D u+\bg,\vec{e}_n}|^pd\mu\right)^{1/p}\\
    &\qquad\le C\left(\dashint_{B_{2d}^+(\bar x_0)}|\mean{A\D u+\bg,\vec{e}_n}|^pd\mu\right)^{1/p}\le C\omega({\color{blue}2}d)\le C\omega(2^kr).
\end{align*}
Combining the preceding two estimates gives
\begin{align*}
    \left(\dashint_{B_r(x_0)}|U^{x_0}|^pd\mu\right)^{1/p}&\le Ck^{\frac{1-p}p}\sum_{j=1}^k\omega(2^jr)\le Ck^{\frac{1-p}p}\sum_{j=1}^k(2^{j\be}\omega(r))\\
    &\le Ck^{\frac{1-p}p}2^{k\be}\omega(r)\le C2^k\omega(r)\le \frac{Cd}r\omega(r).
\end{align*}
This, together with \eqref{eq:U-diff-est}, concludes
$$
\left(\dashint_{B_r(x_0)}|U^{x_0}-U^{y_0}|^pd\mu\right)^{1/p}\le C\omega(r).
$$
This completes the proof.
\end{proof}

\subsection{\texorpdfstring{$C^k$}{} estimates}

In this subsection, we establish Theorem~\ref{thm:deg-pde-HO} by employing the case $k=1$ (Theorem~\ref{thm:deg-pde}) and the induction argument. 

In the statement of Theorem~\ref{thm:deg-pde-HO}, we assume $D_{x'}^{k-1}\bg\in C^{0,\omega}_{1,\mu}(B_1^+)$ (same for $A$), which is a weaker requirement than $D^{k-1}\bg\in C^{0,\omega}_{1,\mu}(B_1^+)$. While this may not be a significant improvement, this formulation of the theorem is crucial for facilitating the induction argument.

In addition, one can infer that the modulus $\sigma(r)$ of the latter result in Theorem~\ref{thm:deg-pde-HO} is comparable with $\int_0^r\frac{\hat\omega(s)}sds+r^\be$ for any chosen $\be\in(0,1)$. This implies that our result recovers the classical $C^{k,\gamma}$ estimates in \cite{TerTorVit22} when the data belong to the H\"older space.

\begin{proof}[Proof of Theorem~\ref{thm:deg-pde-HO}]
As \eqref{Neumann} follows by taking $\rho\to0$ in \eqref{eq:deg-phi-est-bdry}, it is sufficient to establish the $C^k$ estimates of the solution $u$. %We prove them by following the argument in \cite{TerTorVit22}*{Lemma~2.7}.

We argue by induction on $k\in\mathbb N$. The case $k=1$ follows from Theorem~\ref{thm:deg-pde}. We now assume the theorem is true for $k$ and prove it for $k+1$. We observe that the tangential derivatives $u_i:=\partial_iu$ for $i=1,\ldots,n-1$, solve
$$
-\div(x_n^\al A\D u_i)=\div(x_n^\al(\partial_i\bg+\partial_iA\D u))\quad\text{in }B_1^+,
$$
with the conormal boundary condition on $B_1'$
$$
\lim_{x_n\to0^+}x_n^\al\mean{A\D u_i+\partial_i\bg+\partial_iA\nabla u,\vec{e}_n}=0.
$$
As $A, \bg\in C^{k}\subset C^{k-1,1}$, the Schauder estimates in \cite{TerTorVit22}*{Theorem 2.6} give $u\in C^{k,\be}$ for any $0<\be<1$. This implies that the field $\overline\bg :=\partial_i\bg+\partial_iA\D u$ belongs to $C^{k-1}$ with $D_{x'}^{k-1} \overline\bg \in C_{1,\mu}^{0,\omega}$. Thus we have by the inductive hypothesis
\begin{align}\label{eq:tan-deriv-reg}
u_i\in C^{k}_\loc(B_1^+\cup B_1'),\quad \forall \ i=1,\ldots,n-1.
\end{align}
For the $C^{k+1}$-regularity of $u$, it is sufficient to show $u_{nn}:=\partial_{nn}u\in C^{k-1}_\loc(B_1^+\cup B_1')$. To this aim, we rewrite the equation \eqref{eq:deg/sing-pde} as
$$
-\div(A\D u)=\frac{\al\mean{A\D u+\bg,\vec{e}_n}}{x_n}+\div\bg.
$$
It follows that for $U=\mean{A\D u+\bg,\vec{e}_n}$,
$$
x_n^{-\al}\partial_n(x_n^\al U)=h:=-\div\bg+\partial_ng_n-\sum_{i=1}^{n-1}\partial_i(\mean{A\D u,\vec{e}_i}).
$$
Notice that $h\in C^{k-1}$. Since $x_n^\al U=0$ on $B_1'$, this equation gives
$$
U(x',x_n)=\frac1{x_n^\al}\int_0^{x_n}t^\al h(x',t)dt,
$$
and thus
\begin{align*}
    \partial_nU(x',x_n)&=h(x',x_n)-\frac\al{x_n^{\al+1}}\int_0^{x_n}t^\al h(x',t)dt=h(x',x_n)-\al\int_0^1s^\al h(x',sx_n)ds.
\end{align*}
Therefore, $\partial_n U$ belongs to $C^{k-1}$, with its modulus of continuity dominated by that of $h$. The definition of $U$, along with \eqref{eq:tan-deriv-reg}, readily implies $u_{nn}\in C^{k-1}$.
\end{proof}

%%%%%%%%%%%%%%%%%%%%%%%%%%%%%%%%%%%%%%%%%%%%%%%%%%%%%%%%%%%%%%

\section{The Hopf-Oleinik boundary point principle}\label{sec:Hopf}

In this short section, we discuss the validity of the boundary point principle, which holds true under the same conditions stated in \cite{RenSirSoa22} but with a weaker requirement on the boundary regularity of the domain $\Omega$; that is, $\Omega\in C^{1,1-\mathrm{DMO}}$. This is just a remark once one observes that the flattening of such a boundary \eqref{standard_diffeo} leads to the same situation as in \cite{RenSirSoa22}.

Recall that \cite{ApuNaz16} provides counterexamples to the boundary point principle, where the boundaries of the domains are parametrized by convex functions which do not satisfy satisfies the interior $C^{1,\mathrm{Dini}}$-paraboloid condition. However this kind of counterexamples fails the $C^{1,\mathrm{DMO}}$ regularity since one can prove the following fact: if $\Omega$ is convex and $C^{1,1-\mathrm{DMO}}$, then $\Omega$ satisfies the interior $C^{1,\mathrm{Dini}}$-paraboloid condition \cite{ApuNaz19}; that is, the following result holds true.
\begin{proposition}\label{convexDMO}
    Let $\varphi$ be a convex function in $B_1$ such that $\varphi(0)=0$ and $\nabla\varphi(0)=0$. If $\nabla\varphi$ is of $L^1$-DMO in $B_1$, then
    $$\omega(r)=\sup_{|x|\leq r} \frac{\varphi(x)}{|x|}$$
    is a Dini function.
\end{proposition}
\begin{proof}
    Let us consider the ball $B_r$ and select a direction, without loss of generality $\vec e_n$. Taking $0<s\leq r$, by convexity
    \begin{equation}\label{convex1}
    \frac{\varphi(s\vec e_n)}{s}\leq \frac{\varphi(r\vec e_n)}{r} \leq \frac{1}{r}\dashint_{B_r(r\vec e_n)\cap\{x_n=r\}}\varphi(y,r) \, dy,
    \end{equation}
    where $y=(y_1,\ldots,y_{n-1})$. Then, for any $y\in B_r(r\vec e_n)\cap\{x_n=r\}$ we have
    \begin{align*}
    \vp(y,r)&=\vp(\la y,\la r)|^1_0=\int_0^1(y,r)\cdot\D\vp(\la y,\la r)\,d\la=\sqrt{|y|^2+r^2}\int_0^1\D_{\vec v}\vp(\la y,\la r)\,d\la,
    \end{align*}
where $\vec v=\frac{(y,r)}{\sqrt{|y|^2+r^2}}$. % $$\frac{\varphi(y,r)}{\sqrt{|y|^2+r^2}}=\int_0^1D_{\vec v}\varphi(\lambda y,\lambda r) \, d\lambda,\qquad\mathrm{with \ } \vec v=\frac{(y,r)}{\sqrt{|y|^2+r^2}}.$$
    By the convexity of $\varphi$, $D_{\vec v}\varphi(\lambda y,\lambda r)$ is nonnegative whenever $\lambda>0$ and nonpositive whenever $\lambda<0$. Thus,
\begin{align*}
    \frac{\vp(y,r)}r&\le C\int_0^1|\D_{\vec v}\vp(\la y,\la r)|\,d\la\le C\int_0^1|\nabla_{\vec v}\varphi(\lambda y,\lambda r)-\nabla_{\vec v}\varphi(-\lambda y,-\lambda r)| \, d\lambda\\
    &\le C\int_0^1|\nabla\varphi(\lambda y,\lambda r)-\nabla\varphi(-\lambda y,-\lambda r)| \, d\lambda.
\end{align*}
    %$$\frac{\varphi(y,r)}{\sqrt{|y|^2+r^2}}\leq\int_{-1}^0 d\tilde\lambda\int_0^1(D_{\vec v}\varphi(\lambda y,\lambda r)-D_{\vec v}\varphi(\tilde\lambda y,\tilde\lambda r)) \, d\lambda.$$
    %\rred{Then
    %$$\frac{\varphi(y,r)}{r}\leq C\int_0^1|\nabla\varphi(\lambda y,\lambda r)-\nabla\varphi(-\lambda y,-\lambda r)| \, d\lambda.$$}
     Then, considering \eqref{convex1}
     \begin{eqnarray*}
         \frac{\varphi(s\vec e_n)}{s}&\leq& \frac{1}{r}\dashint_{B_r(r\vec e_n)\cap\{x_n=r\}}\varphi(y,r) \, dy\\
         &\leq& C\dashint_{B_r(\vec e_n)\cap\{x_n=r\}}\int_0^1|\nabla\varphi(\lambda y,\lambda r)-\nabla\varphi(-\lambda y,-\lambda r)| \, dy \, d\lambda\\
         &\leq& C \dashint_{D_r}|\nabla\varphi(x)-\nabla\varphi(-x)| \, dx\\
         &\leq& C \dashint_{D_r}\dashint_{D_r}|\nabla\varphi(x)-\nabla\varphi(z)|+|\nabla\varphi(z)-\nabla\varphi(-x)| \, dx \, dz\\
         &\leq& C \eta_{\nabla\varphi}^1(r).
     \end{eqnarray*}
     The domain $D_r$ above stands for a portion of a cone
     $$D_r=\{\lambda y \ : \ \lambda\in(0,1), \, y\in B_r(\vec e_n)\cap\{x_n=r\}\}.$$
     The latter bound is uniform in the choice of the direction, then
\begin{equation*}\sup_{|x|\leq r} \frac{\varphi(x)}{|x|}\leq C\eta_{\nabla\varphi}^1(r).\qedhere
\end{equation*}
\end{proof}

%%%%%%%%%%%%%%%%%%%%%%%%%%%%%%%%%%%%%%%%%%%%%%%%%%%%%%%%%%%%%%

\section{Higher-order boundary Harnack principle}\label{sec:BHP}
The Schauder type estimates derived in Section~\ref{3}, particularly the case $\al=2$, can be utilized to establish higher-order boundary Harnack principles, as shown in \cite{TerTorVit22} within the framework of H\"older condition. Recently, \cite{JeoVit23} demonstrated the applicability of this technique under the uniform Dini condition. However, as previously mentioned, the situation becomes more complicated within our DMO framework.

\subsection{Dini mean oscillation of \texorpdfstring{$u/x_n$}{}}
The purpose of this section is to establish the following result.
\begin{proposition}
    \label{prop:ratio-L1-DMO}
    Let $u\in H^1(B_1^+)$ be a weak solution of 
\begin{align}\label{flatUEeq}
    \begin{cases}
    -\div(A\D u)=\div\bg&\text{in }B_1^+,\\
    u=0&\text{on }B_1'.
    \end{cases}
\end{align}
If $A$ and $\bg$ are of $L^1(dx)$-DMO in $B_1^+$, then $u/x_n$ is of $L^1(x_ndx)$-DMO in $B^+_{1/2}$.
\end{proposition}

In the subsequent subsection, we will utilize Proposition~\ref{prop:ratio-L1-DMO} to obtain the higher regularity of $u/x_n$, Corollary~\ref{cor:HO-u-x_n-reg}, which will play a significant role in establishing the higher-order boundary Harnack principle. Notice that the equation \eqref{flatUEeq} differs from the one \eqref{eq:UE-Diri} in Corollary~\ref{cor:HO-u-x_n-reg}. We start with \eqref{flatUEeq} to facilitate the induction process in Proposition~\ref{prop:ratio-L1-DMO-k}.

In the remaining of this subsection, we fix
$$
d\mu=x_ndx
$$
and define
$$
\phi(x_0,r):=\dashint_{B_r(x_0)\cap B_1^+}\left|\frac{u}{x_n}-\lmean{\frac{u}{x_n}}_{B_r(x_0)\cap B_1^+}^\mu\right|d\mu,\quad x_0\in \overline{B_{1/2}^+},\, r>0.
$$
% where
% $$
% \lmean{\frac{u}{x_n}}^\mu_{B_r(x_0)\cap B_1^+}=\dashint_{B_r(x_0)\cap B_1^+}\frac{u}{x_n}d\mu=\frac{\int_{B_r(x_0)\cap B_1^+}\frac{u}{x_n}d\mu}{\int_{B_r(x_0)\cap B_1^+}d\mu}.
% $$
We first prove the following auxiliary results.

\begin{lemma}
    \label{lem:ratio-L1-DMO-bdry}
Let $u,A, \bg$ be as in Proposition~\ref{prop:ratio-L1-DMO}. Then for any $\bar x_0\in B'_{1/2}$ and $0<\rho<r\le 1/4$,
\begin{align}\label{eq:weight-phi-L1-est}
\phi(\bar x_0,\rho)\le C(\rho/r)^{1/2}\phi(\bar x_0,r)+C\|\D u\|_{L^\infty(B_1^+)}\tilde\eta_A^{1}(\rho)+C\tilde\eta_\bg^{1}(\rho),
\end{align}
where $C=C(n,\la)>0$ are constants, and $\tilde\eta_\bullet^{1}$ is a Dini function derived from $\eta^{1}_\bullet$.
\end{lemma}

\begin{proof}
We may assume without loss of generality that $\bar x_0=0$. We fix $r\in (0,1/2)$, and write for simplicity $\bar A=\mean{A}_{B_r^+}$ and $\bar \bg=\mean{\bg}_{B_r^+}$. Let $w\in W^{1,2}_0(B_r^+)$ be a solution of
$$
-\div(\bar A\D w)=\div((A-\bar A)\D u+(\bg-\bar\bg))\quad\text{in }B_r^+.
$$
By using the estimate of Green's functions, we have
\begin{align}\label{eq:sol-hom-diff-L1-est-1}\begin{split}
    \int_{B_r^+}|w(x)|\,dx&\le Cr\int_{B_r^+}|(A-\bar A)\D u+(\bg-\bar\bg)|\,dx\\
    &\le Cr^{n+1}\left(\|\D u\|_{L^\infty(B_1^+)}\eta_A^{1}(r)+\eta_\bg^{1}(r)\right).
\end{split}\end{align}
This estimate implies that for a small constant $\ka\in(0,1/2)$ to be chosen later
\begin{align}
    \label{eq:sol-hom-diff-L1-est-2}
    \begin{split}
        \dashint_{B_{\ka r}^+}\left|\frac{w}{x_n}-\lmean{\frac{w}{x_n}}_{B_{\ka r}^+}^\mu\right|d\mu&\le 2\dashint_{B_{\ka r}^+}\left|\frac{w}{x_n}\right|d\mu\le \frac{2\int_{B_r^+}|w(x)|dx}{\mu(B_{\ka r}^+)}\\
        &\le C\ka^{-(n+1)}\left(\|\D u\|_{L^\infty(B_1^+)}\eta_A^{1}(r)+\eta_\bg^{1}(r)\right).
    \end{split}
\end{align}
On the other hand, $v:=u-w$ solves
\begin{align*}
    \begin{cases}
        \div(\bar A\D v)=0&\text{in }B_r^+,\\
        v=u&\text{on }\partial B_r^+.
    \end{cases}
\end{align*}
Note that $v=u=0$ on $B'_r$. From  $v(x)=\int_0^{x_n}\partial_nv(x',t)dt=x_n\int_0^1\partial_nv(x',x_ns)ds$, we infer by using the boundary elliptic estimate
\begin{align*}
    \left[\frac{v}{x_n}\right]_{\text{Lip}(B^+_{r/2})}&\le [\D v]_{\text{Lip}(B_{r/2}^+)}=\|D^2v\|_{L^\infty(B^+_{r/2})}\le \frac{C}{r^2}\dashint_{B_r^+}|v(x)|dx\\
    &\le \frac{C}r\cdot\frac{\int_{B_r^+}\left|\frac{v(x)}{x_n}\right|x_n\,dx}{\mu(B_r^+)}=\frac{C}{r}\dashint_{B_r^+}\left|\frac{v}{x_n}\right|d\mu.
\end{align*}
For every constant $c\in \R$, as the function $\tilde v(x)=v(x)-cx_n$ is a solution of $\div(\bar A\D\tilde v)=0$ in $B_r^+$ with $\tilde v=v=0$ on $B'_r$, repeating the above process with $\tilde v$ in the place of $v$ yields
$$
\left[\frac{v}{x_n}\right]_{\text{Lip}(B^+_{r/2})}\le \frac{C}r\dashint_{B_r^+}\left|\frac{v}{x_n}-c\right|d\mu.
$$
This inequality implies by setting $c=\lmean{\frac{v}{x_n}}_{B_r^+}^\mu$,
\begin{align}\label{eq:hom-replace-est}
\begin{split}
    \dashint_{B_{\ka r}^+}\left|\frac{v}{x_n}-\lmean{\frac{v}{x_n}}_{B_{\ka r}^+}^\mu\right|d\mu
    &\le C\ka r\left[\frac{v}{x_n}\right]_{\text{Lip}(B^+_{r/2})}\\
    &\le C\ka\dashint_{B_r^+}\left|\frac{v}{x_n}-\lmean{\frac{v}{x_n}}_{B_r^+}^\mu\right|d\mu.
\end{split}\end{align}
This, together with \eqref{eq:sol-hom-diff-L1-est-1}, gives
\begin{align*}
    &\dashint_{B_{\ka r}^+}\left|\frac{v}{x_n}-\lmean{\frac{v}{x_n}}_{B^+_{\ka r}}^\mu\right|d\mu\\
    &\qquad\le C\ka\dashint_{B_r^+}\left|\frac{u}{x_n}-\lmean{\frac{u}{x_n}}_{B_r^+}^\mu\right|d\mu+C\ka\dashint_{B_r^+}\left|\frac{w}{x_n}-\lmean{\frac{w}{x_n}}_{B_r^+}^\mu\right|d\mu\\
    &\qquad\le  C\ka\dashint_{B_r^+}\left|\frac{u}{x_n}-\lmean{\frac{u}{x_n}}_{B_r^+}^\mu\right|d\mu+C\ka\dashint_{B_r^+}\left|\frac{w}{x_n}\right|d\mu\\
    &\qquad\le C\ka\dashint_{B_r^+}\left|\frac{u}{x_n}-\lmean{\frac{u}{x_n}}_{B_r^+}^\mu\right|d\mu+C\ka\left(\|\D u\|_{L^\infty(B_1^+)}\eta_A^{1}(r)+\eta_\bg^{1}(r)\right).
\end{align*}
By combining this estimate with \eqref{eq:sol-hom-diff-L1-est-2}, we  obtain that
\begin{align*}
    \phi(0,\ka r)&\le \dashint_{B_{\ka r}^+}\left|\frac{v}{x_n}-\lmean{\frac{v}{x_n}}_{B^+_{\ka r}}^\mu\right|d\mu+\dashint_{B_{\ka r}^+}\left|\frac{w}{x_n}-\lmean{\frac{w}{x_n}}_{B^+_{\ka r}}^\mu\right|d\mu\\
    &\le C\ka\phi(0,r)+C\ka^{-(n+1)}\left(\|\D u\|_{L^\infty(B_1^+)}\eta_A^{1}(r)+\eta_\bg^{1}(r)\right),
\end{align*}
where constants $C>0$ depend only on $n$ and $\la$. As before, we can choose $\ka$ small so that $C\ka\le \ka^{1/2}$ and use iteration argument to deduce \eqref{eq:weight-phi-L1-est}.
\end{proof}

\begin{lemma}
    \label{lem:ratio-L1-DMO-interior}
Let $u, A, \bg$ be as in Proposition~\ref{prop:ratio-L1-DMO}. For any $x_0=(x_0',(x_0)_n)\in B_{1/2}^+$ and $0<\rho<r<\frac12(x_0)_n$, the estimate \eqref{eq:weight-phi-L1-est} holds true.
\end{lemma}

\begin{proof}
We follow the argument in Lemma~\ref{lem:ratio-L1-DMO-bdry}. To begin with, we decompose $u=w+v$ in $B_r(x_0)\subset B_1^+$, where $w\in W^{1,2}_0(B_r(x_0))$ is a solution of
$$
-\div(\mean{A}_{B_r(x_0)}\D w)=\div((A-\mean{A}_{B_r(x_0)})\D u+(\bg-\mean{\bg}_{B_r(x_0)}))\quad\text{in }B_{r}(x_0),
$$
and $v$ is a solution to
\begin{align}\label{eq:hom-repl}
\div(\mean{A}_{B_r(x_0)}\D v)=0\quad\text{in }B_r(x_0)
\end{align}
with $v-u\in W_0^{1,2}(B_r(x_0))$. For $d:=(x_0)_n$, the condition $r<d/2$ implies that $d/2<x_n<3d/2$ whenever $x\in B_r(x_0)$, which yields for small $\ka\in (0,1/2)$
$$
\mu(B_r(x_0))\le Cr^nd \quad\text{and}\quad \mu(B_{\ka r}(x_0))\ge c(\ka r)^nd\ge c\ka^n r^{n+1}.
$$
By arguing as in Lemma~\ref{lem:ratio-L1-DMO-bdry}, we can get
$$
\int_{B_r(x_0)}|w(x)|dx\le Cr^{n+1}\left(\|\D u\|_{L^\infty(B_1^+)}\eta_A^1(r)+\eta_\bg^1(r)\right),
$$
which implies an analogue of \eqref{eq:sol-hom-diff-L1-est-2}:
\begin{align}
    \label{eq:w/x_n-L1-est}
    \begin{split}
        \dashint_{B_{\ka r}(x_0)}\left|\frac{w}{x_n}-\lmean{\frac{w}{x_n}}_{B_{\ka r}(x_0)}^\mu\right|d\mu
        &\le \frac{2\int_{B_r(x_0)}|w(x)|dx}{\mu(B_{\ka r}(x_0))}\\
        &\le C\ka^{-n}\left(\|\D u\|_{L^\infty(B_1^+)}\eta_A^{1}(r)+\eta_\bg^{1}(r)\right).
    \end{split}
\end{align}
Concerning $\frac{v}{x_n}$, we can exploit $L^\infty$-estimates for $v$ and $\D v$ to deduce
\begin{align*}
    \left[\frac{v}{x_n}\right]_{\text{Lip}(B_{r/2}(x_0))}&=\left[\D\left(\frac{v}{x_n}\right)\right]_{L^\infty(B_{r/2}(x_0))}\\
    &\le C\left(\frac{\|\D v\|_{L^\infty(B_{r/2}(x_0))}}{d}+\frac{\|v\|_{L^\infty(B_{r/2}(x_0))}}{d^2}\right)\\
    &\le \frac{C}{r^{n+1}d}\int_{B_r(x_0)}|v(x)|\,dx\le \frac{C}r\dashint_{B_r(x_0)}\left|\frac{v}{x_n}\right|d\mu.
\end{align*}
Replacing $v$ with $\tilde v(x):=v(x)-\lmean{\frac{v}{x_n}}_{B_r(x_0)}^\mu x_n$, which also satisfies \eqref{eq:hom-repl}, gives
$$
\left[\frac{v}{x_n}\right]_{\text{Lip}(B_{r/2}(x_0))}\le \frac{C}r\dashint_{B_r(x_0)}\left|\frac{v}{x_n}-\lmean{\frac{v}{x_n}}_{B_r(x_0)}^\mu\right|d\mu.
$$
Thus
\begin{align}
    \label{eq:v/x_n-est-L1}
    \dashint_{B_{\ka r}(x_0)}\left|\frac{v}{x_n}-\lmean{\frac{v}{x_n}}_{B_{\ka r}(x_0)}^\mu\right|d\mu\le C\ka\dashint_{B_r(x_0)}\left|\frac{v}{x_n}-\lmean{\frac{v}{x_n}}_{B_r(x_0)}^\mu\right|d\mu.
\end{align}
This is an analogue of \eqref{eq:hom-replace-est}. As we have seen in the proof of Lemma~\ref{lem:ratio-L1-DMO-bdry}, the estimates \eqref{eq:w/x_n-L1-est} and \eqref{eq:v/x_n-est-L1} imply
$$
\phi(x_0,\ka r)\le C\ka\phi(x_0,r)+C\ka^{-n}\left(\|\D u\|_{L^\infty(B_1^+)}\eta_A^{1}(r)+\eta_\bg^{1}(r)\right).
$$
This concludes the lemma by choosing $\ka$ sufficiently small and using the iteration argument.
\end{proof}

We now provide the proof of Proposition~\ref{prop:ratio-L1-DMO} with the help of Lemmas~\ref{lem:ratio-L1-DMO-bdry} and \ref{lem:ratio-L1-DMO-interior}.

\begin{proof}[Proof of Proposition~\ref{prop:ratio-L1-DMO}]
We consider
$$
\omega(r):=\|\D u\|_{L^\infty(B_1^+)}\sup_{s\in [r,1]}\left[(r/s)^{1/2}\tilde\eta_A^{1}(s)\right]+\sup_{s\in [r,1]}\left[(r/s)^{1/2}\tilde\eta_\bg^{1}(s)\right],\quad 0<r<1.
$$
Note that $\omega(r)\ge\|\D u\|_{L^\infty(B_1^+)}\tilde\eta_A^{1}(r)+\tilde\eta_\bg^{1}(r)$ and  $r\longmapsto\frac{\omega(r)}{r^{1/2}}$ in nonincreasing. We claim that for any $x_0\in B_{1/2}^+$ and $0<\rho<r\le 1/4$
\begin{align}
    \label{eq:phi-est}
    \phi(x_0,\rho)\le C(\rho/r)^{1/2}\phi(x_0,r)+C\omega(\rho)
\end{align}
for some constant $C=C(n,\la)>0$. As $\frac{u}{x_n}$ is bounded, \eqref{eq:phi-est} readily implies Proposition~\ref{prop:ratio-L1-DMO} by taking $r=1/4$. Before we prove \eqref{eq:phi-est}, we observe that if $B_s(z_0)\subset B_t(z_1)$ and $\mu(B_s(z_0)\cap B_1^+)\ge c_0\mu(B_t(z_1)\cap B_1^+)$ for some $0<c_0<1$, then $\phi(z_0,s)\le 2c_0^{-1}\phi(z_1,t)$. This will be used multiple times in the proof, and follows from the following computation:
\begin{align*}
    &\dashint_{B_s(z_0)}\left|\frac{u}{x_n}-\lmean{\frac{u}{x_n}}^\mu_{B_s(z_0)}\right|d\mu\\
    &\qquad \le \dashint_{B_s(z_0)}\left|\frac{u}{x_n}-\lmean{\frac{u}{x_n}}^\mu_{B_t(z_1)}\right|d\mu+\left|\lmean{\frac{u}{x_n}}^\mu_{B_s(z_0)}-\lmean{\frac{u}{x_n}}_{B_t(z_1)}^\mu\right|\\
    &\qquad\le 2\dashint_{B_s(z_0)}\left|\frac{u}{x_n}-\lmean{\frac{u}{x_n}}^\mu_{B_t(z_1)}\right|d\mu\le \frac{2}{c_0}\dashint_{B_t(z_1)}\left|\frac{u}{x_n}-\lmean{\frac{u}{x_n}}^\mu_{B_t(z_1)}\right|d\mu.
\end{align*}

To derive \eqref{eq:phi-est}, we fix a point $x_0=(x_0',(x_0)_n)\in B_{1/2}^+$, and write $d:=(x_0)_n$ and $\bar x_0:=(x_0',0)\in B_{1/2}'$. We split our proof into two cases
$$
\text{either }\,\,\, \rho\ge d/2\,\,\, \text{ or  }\,\,\,\rho<d/2.
$$

\medskip\noindent\emph{Case 1.} We first consider the case $\rho\ge d/2$.

\medskip\noindent\emph{Case 1.1.} Suppose $\rho<r/6$. By using the observation above, we can obtain
$$
\phi(x_0,\rho)\le C\phi(\bar x_0,3\rho)\quad\text{and}\quad \phi(\bar x_0,r/2)\le C\phi(x_0,r).
$$
We then have by Lemma~\ref{lem:ratio-L1-DMO-bdry},
\begin{align*}
    \phi(x_0,\rho)&\le C\phi(\bar x_0,3\rho)\le C\left(\frac{3\rho}{r/2}\right)^{1/2}\phi(\bar x_0,r/2)+C\omega(3\rho)\\
    &\le C(\rho/r)^{1/2}\phi(x_0,r)+C\omega(\rho).
\end{align*}

\medskip\noindent\emph{Case 1.2.} If $r/6\le \rho<r$, then we simply have by using the above observation
$$
\phi(x_0,\rho)\le C\phi(x_0,r)\le C(\rho/r)^{1/2}\phi(x_0,r).
$$

\medskip\noindent\emph{Case 2.} Suppose $\rho<d/2$. If $r<d/2$, then \eqref{eq:phi-est} simply follows from Lemma~\ref{lem:ratio-L1-DMO-interior}. Thus we may assume $\rho<d/2\le r$. Notice that by Lemma~\ref{lem:ratio-L1-DMO-interior} again,
\begin{align}
    \label{eq:phi-est-d}
    \phi(x_0,\rho)\le C(\rho/d)^{1/2}\phi(x_0,d/2)+C\omega(\rho).
\end{align}
We consider further subcases either $r/4\le d$ or $d<r/4$.

\medskip\noindent\emph{Case 2.1.} Suppose $d/8<r/4\le d$. Then we readily have $\phi(x_0,d/2)\le C\phi(x_0,r)$, which combined with \eqref{eq:phi-est-d} yields
$$
\phi(x_0,\rho)\le C(\rho/d)^{1/2}\phi(x_0,r)+C\omega(\rho)\le C(\rho/r)^{1/2}\phi(x_0,r)+C\omega(\rho).
$$

\medskip\noindent\emph{Case 2.2.} It remains to consider the case $\rho<d/2<r/8$. By applying Lemma~\ref{lem:ratio-L1-DMO-bdry}, we infer
\begin{align*}
    \phi(x_0,d/2)&\le C\phi(\bar x_0,2d)\le C(d/r)^{1/2}\phi(\bar x_0,r/2)+\omega(2d)\\
    &\le C(d/r)^{1/2}\phi(x_0,r)+\omega(2d).
\end{align*}
This, along with \eqref{eq:phi-est-d} and the monotonicity of $t\longmapsto \frac{\omega(t)}{t^{1/2}}$, gives
\begin{align*}
\phi(x_0,\rho)&\le C(\rho/r)^{1/2}\phi(x_0,r)+C(\rho/d)^{1/2}\omega(2d)+C\omega(\rho)\\
&\le C(\rho/r)^{1/2}\phi(x_0,r)+C\omega(\rho).
\end{align*}
This completes the proof.
\end{proof}

%%%%%%%%%%%%%%%%%%%%%%%%%%%%%%%%%%%%%%%%%%%%%%%%%%%%%%%%%%%%%%%%%%

\subsection{Higher-order boundary Harnack principle on a fixed boundary}\label{4}
The aim of this subsection is to establish the higher-order boundary Harnack principle, Theorem~\ref{thm:BHP}. In fact, the most technical part of the proof has already been done in the previous subsection, where we proved the Dini mean oscillation of $u/x_n$ in Proposition~\ref{prop:ratio-L1-DMO}. In this subsection, we first extend that result to deduce the higher regularity of $u/x_n$; see Corollary~\ref{cor:HO-u-x_n-reg}. We then follow the argument presented in \cite{TerTorVit22}*{Theorem~1.2} to achieve Theorem~\ref{thm:BHP}.

To extend Proposition~\ref{prop:ratio-L1-DMO}, we need the following auxiliary result.

\begin{proposition}\label{flatschauder1}
Let $n\geq2$, $k\in\mathbb N$, $\omega$ a Dini function and $u$ be a weak solution to \eqref{flatUEeq}:
\begin{equation*}
    \begin{cases}
        -\div(A\D u)=\div\bg&\text{in }B_1^+,\\
        u=0&\text{on }B_1'.
    \end{cases}
    \end{equation*}
Let $A,\bg\in C^{k-1}(\overline{B_1^+})$ with $D^{k-1}_{x'}A,D^{k-1}_{x'}\bg\in C^{0,\omega}_{1}(B_1^+)$. Then $u\in C^{k}_\loc(B_1^+\cup B_1')$. Moreover, if $\|A\|_{C^{k-1}(B_1^+)}+\sum_{|\beta|=k-1}[D_{x'}^{\beta}A]_{C_{1,\mu}^{0,\omega}(B_1^+)}\leq L$, then there exists a positive constant $C$ and a modulus of continuity $\sigma$ depending only on $n,\lambda,k,L$ and $\omega$ such that
\begin{equation}\label{eq:flat-Sch-est}
\|u\|_{C^{k,\sigma}(B^+_{1/2})}\leq C\left(\|u\|_{L^{2}(B_1^+)}+\|\bg\|_{C^{k-1}(B_1^+)}+\sum_{|\beta|=k-1}[D_{x'}^{\beta}\bg]_{C_{1}^{0,\omega}(B_1^+)}\right).
\end{equation}
\end{proposition}

\begin{proof}
Let us prove the result by induction on $k\in\mathbb N$. The case $k=1$ is \cite{DonEscKim18}*{Proposition 2.7}. Let us suppose the result is true for a certain $k\in\mathbb N$ and prove it for $k+1$. Assuming $A,\bg\in C^k(\overline{B_1^+})$ and $D^k_{x'}A,D^k_{x'}\bg\in C_1^{0,\omega}(B_1^+)$, we want to prove $u\in C^{k+1}_\loc(B_1^+\cup B_1')$, which is equivalent to prove $u_i=\partial_iu\in C^{k}_\loc(B_1^+\cup B_1')$ for any $i=1,\ldots,n$. It is easily seen that any tangential derivative $u_i=\partial_i u$ with $i=1,\ldots,n-1$ is a solution to
\begin{equation}\label{eq:tangentialderivative}
 \begin{cases}
 -\div(A\D u_i)=\div (\partial_i\bg+\partial_iA\nabla u)   & \text{in }B^+_{1}\\
 u_i=0 & \text{on }B'_{1}.
 \end{cases}
 \end{equation}
 Hence
 \begin{equation}\label{u_iL1}
 u_i\in C^{k}_\loc(B_1^+\cup B_1') \qquad\mathrm{for \  any \ } i=1,\ldots,n-1
 \end{equation}
 by the inductive hypothesis since $\overline\bg:=\partial_i\bg+\partial_iA\nabla u\in C^{k-1}$ with $D^{k-1}_{x'}\overline\bg \in C_1^{0,\omega}$. Let us remark that we used that $A,\bg\in C^k\subset C^{k-1,1}$ and by standard Schauder estimates $u\in C^{k,\beta}$ for any $0<\beta<1$, i.e. $\nabla u \in C^{k-1,\beta}$ for any $0<\beta<1$.

 In order to prove that $u_n\in C^{k}_\loc(B_1^+\cup B_1')$, it is enough to prove that $u_{nn}=\partial^2_{nn}u\in C^{k-1}_\loc(B_1^+\cup B_1')$ because $u_{ni}:=\partial^2_{ni}u\in C^{k-1}_\loc(B_1^+\cup B_1')$ for any $i=1,\ldots,n-1$ was already given by \eqref{u_iL1}. Then, for this last partial derivative, one can rewrite equation \eqref{flatUEeq} as
$$
 u_{nn}=-\frac{1}{a_{nn}}\left(\div \bg+\sum_{i=1}^{n-1}\partial_i(\mean{A\nabla u,\vec e_i})+\sum_{i=1}^{n-1}\partial_n(a_{ni}u_i)+\partial_na_{nn}u_n\right)\in C^{k-1}.
$$
Observe that $a_{nn}=\mean{A\vec e_n,\vec e_n}\geq \lambda>0$.
\end{proof}

%\begin{remark}
 %   By employing the standard flattening of the boundary \eqref{standard_diffeo}, Proposition~\ref{flatschauder1} can be generalized to the boundary Schauder estimate in $C^{k,\text{DMO}}$ domains as follows: let $u$ be a solution of 
  %  \begin{align*}
   %     \begin{cases}
    %        -\div(A\D u)=\div\bg&\text{in }\Omega\cap B_1,\\
     %       u=0&\text{on }\partial\Omega\cap B_1.
      %  \end{cases}
   % \end{align*}
    %If $\Omega$ is $C^{k,1-\text{DMO}}$ and $A,\bg\in C^{k-1,1-\text{DMO}}(\Omega\cap B_1)$, then $u\in C^{k}_{\loc}(\overline{\Omega}\cap B_1)$ and the similar estimate as \eqref{eq:flat-Sch-est} holds.
%\end{remark}

By employing the standard flattening of the boundary \eqref{standard_diffeo}, Proposition~\ref{flatschauder1} can be generalized to the boundary Schauder estimate in $C^{k,1-\text{DMO}}$ domains as follows: let $u$ be a solution of 
    \begin{align}\label{eqDMOBoundary}
        \begin{cases}
            -\div(A\D u)=\div\bff&\text{in }\Omega\cap B_1,\\
            u=0&\text{on }\partial\Omega\cap B_1.
        \end{cases}
    \end{align}
    Then the following result holds true
\begin{corollary}[Boundary Schauder estimates in $C^{k,1-\mathrm{DMO}}$ domains]\label{CkUnif1}
Let $n\geq2$, $k\in\mathbb N$, $\omega$ a Dini function, and $u$ be a weak solution to \eqref{eqDMOBoundary}. Let $A,\bff\in C^{k-1,\omega}_{1}(\Omega\cap B_1)$, $\gamma\in C^{k,\omega}_{1}(B_1')$. Then $u\in C^{k}_\loc(\overline\Omega\cap B_1)$. Moreover, if $\|A\|_{C^{k-1,\omega}_{1}(\Omega\cap B_1)}+\|\gamma\|_{C^{k,\omega}_{1}(B'_1)}\leq L$, then there exists a positive constant $C$ and a modulus of continuity $\sigma$ depending only on $n,\lambda,k,L$, and $\omega$ such that
\begin{equation*}
\|u\|_{C^{k,\sigma}(\Omega\cap B_{1/2})}\leq C\left(\|u\|_{L^{2}(\Omega\cap B_1)}+\|\bff\|_{C^{k-1,\omega}_{1}(\Omega\cap B_1)}\right).
\end{equation*}
\end{corollary}
Let us remark here that the modulus of continuity $\sigma(r)$ is comparable with $\int_0^r\frac{\hat\omega(s)}{s} \, ds \, +r^\beta$ for any chosen $\beta\in(0,1)$, where $\hat\omega$ is a Dini function derived from $\omega$ as in \eqref{eq:hat-eta}.

We combine Propositions~\ref{prop:ratio-L1-DMO} and \ref{flatschauder1} to derive the following result.

\begin{proposition}
    \label{prop:ratio-L1-DMO-k}
    Let us assume the same conditions of Proposition \ref{flatschauder1}. Then $u/x_n\in C^{k-1}_\loc(B_1^+\cup B_1')$ and $D^{k-1}_{x'}\left(\frac{u}{x_n}\right)$ is of $L^1(x_ndx)$-DMO in $B^+_{1/2}$.
\end{proposition}

\begin{proof}
The fact that $u/x_n\in C^{k-1}_\loc(B_1^+\cup B_1')$ follows by Proposition \ref{flatschauder1} and the fact that
$$\frac{u(x',x_n)}{x_n}=\int_0^{1}\partial_n u(x',sx_n) \,ds.$$
Let us prove that $D^{k-1}_{x'}\left(\frac{u}{x_n}\right)$ is of $L^1(x_ndx)$-DMO by induction on $k\in\mathbb N$. The case $k=1$ is given in Proposition~\ref{prop:ratio-L1-DMO}. Let us suppose the result true for a certain $k\in\mathbb N$ and prove it for $k+1$. Assuming $A,\bg\in C^k(\overline{B_1^+})$ and $D^{k}_{x'}A,D^{k}_{x'}\bg\in C_1^{0,\omega}(B_1^+)$, we want to prove $D^{k}_{x'}\left(\frac{u}{x_n}\right)$ is of $L^1(x_ndx)$-DMO in $B^+_{1/2}$. Let us remark that Proposition~\ref{flatschauder1} implies $u\in C^{k+1}$ and it remains to show that $D^{k-1}_{x'}\left(\frac{\partial_iu}{x_n}\right)$ is of $L^1(x_ndx)$-DMO in $B^+_{1/2}$ for any $i=1,\ldots,n-1$. In the proof of Proposition \ref{flatschauder1} we proved that any tangential derivative $u_i=\partial_i u$ with $i=1,\ldots,n-1$ is a solution to \eqref{eq:tangentialderivative}. Hence the result is proved by the inductive hypothesis since $\overline\bg:=\partial_i\bg+\partial_iA\nabla u\in C^{k-1}$ with $D^{k-1}_{x'}\overline\bg \in C_1^{0,\omega}$.
\end{proof}

\begin{corollary}\label{cor:HO-u-x_n-reg}
    Let $k\ge1$, and suppose that $u$ is a solution of
    \begin{align}\label{eq:UE-Diri}
        \begin{cases}
            -\div(A\D u)=g&\text{in }B_1^+,\\
            u=0&\text{on }B'_1.
        \end{cases}
    \end{align}
    Assume $A, g\in C^{k-1}(\overline{B_1^+})$ and $D_{x'}^{k-1}A\in C_1^{0,\omega}(B_1^+)$ for some Dini function $\omega$. Then the conclusion of Proposition~\ref{prop:ratio-L1-DMO-k} holds.
\end{corollary}

\begin{proof}
In view of Proposition~\ref{prop:ratio-L1-DMO-k}, it is enough to find a function $\bg:B_{3/4}^+\to\R^n$ such that $\div\bg=g$, and $\bg\in C^{k-1}(\overline{B_{3/4}^+})$ with $D_{x'}^{k-1}\bg\in C_1^{0,\omega}(B^+_{3/4})$. To this end, we take a smooth and convex domain $\cD$ such that $B_{3/4}^+\subset \cD\subset B_1^+$, and consider a Dirichlet problem
\begin{align*}
    \begin{cases}
        \Delta w=g&\text{in }\cD,\\
        w=0&\text{on }\partial \cD.
    \end{cases}
\end{align*}
By the elliptic theory, there exists a (unique) solution $w$, which belongs to $C^{k,\be}(\cD)$ for any $0<\be<1$. Then $\bg:=\D w$ is the desired one.    
\end{proof}

Now we are ready to establish Theorem \ref{thm:BHP} by following the argument introduced in \cite{TerTorVit22}*{Theorem~1.2}.

Let us consider two functions $u,v$ solving \eqref{BHconditions}, i.e.,
\begin{equation*}
\begin{cases}
-\div\left(A\nabla v\right)=f &\mathrm{in \ }\Omega\cap B_1,\\
-\div\left(A\nabla u\right)=g &\mathrm{in \ }\Omega\cap B_1,\\
u>0 &\mathrm{in \ }\Omega\cap B_1,\\
u=v=0, \quad \partial_{\nu} u<0&\mathrm{on \ }\partial\Omega\cap B_1,
\end{cases}
\end{equation*}
where $A$ is symmetric and satisfies \eqref{eq:assump-coeffi}, $0\in\partial\Omega$ and $\nu$ stands for the unit outward normal vector to $\Omega$ on $\partial\Omega$.

As shown in \cite{TerTorVit22}, the ratio $w=v/u$ solves the degenerate elliptic equation
\begin{equation*}
-\div\left(u^2A\nabla w\right)=uf-gv \quad\mathrm{in \ }\Omega\cap B_1;
\end{equation*}
with associated conormal boundary condition at $\partial\Omega\cap B_1$, i.e., it is a weak solution in the weighted Sobolev space $H^1(\Omega\cap B_1,u^2(x)dx)$
\begin{equation*}
    \int_{\Omega\cap B_1}u^2A\nabla w \nabla \phi=\int_{\Omega\cap B_1}(uf-vg)\phi
\end{equation*}
for any $\phi\in C^\infty_c(\overline\Omega\cap B_1)$. Theorem \ref{thm:BHP} is proved by composing $u,v$ with the standard diffeomorphism which flattens the boundary, and proving the regularity for the ratio near the flat boundary. Then, the curved world inherits the regularity by composing back with the same diffeomorphism.

Let $k\in\mathbb N$ and $\omega$ be a Dini function. Let us assume that $A,f,g\in C_1^{k-1,\omega}(\Omega\cap B_1)$ and $\partial\Omega\in C^{k,1-\mathrm{DMO}}$.
After rotations and dilations, the domain $\Omega\cap B_1$ can be locally parametrized with $\gamma\in C_1^{k,\omega}$:
$$
\Omega\cap B_{1}=\{x_n>\gamma(x')\}\cap B_{1},\qquad \partial\Omega\cap B_{1}=\{x_n=\gamma(x')\}\cap B_{1}.
$$
Let us consider the standard local diffeomorphism which straightens the boundary $\partial\Omega$:
\begin{align}
    \label{standard_diffeo}
    \Phi(x',x_n)=(x',x_n+\gamma(x')),
\end{align}
which is of class $C^{k,\omega}_1$. By composing $u,v,f,g$ with $\Phi$, one can see that $\tilde v=v\circ\Phi$ and $\tilde u=u\circ\Phi$ solve, up to a further dilation,
\begin{align*}
    \begin{cases}
        -\div(\tilde A\D\tilde v)=\tilde f&\text{in }B_1^+,\\
        -\div(\tilde A\D\tilde u)=\tilde g&\text{in }B_1^+,\\
        \tilde u>0&\text{in }B_1^+,\\
        \tilde u=\tilde v=0,\,\,\, -\partial_n\tilde u<0&\text{on }B_1',
    \end{cases}
\end{align*}
with new free terms $\tilde f=f\circ\Phi$, $\tilde g=g\circ\Phi$ and coefficients $\tilde A=(J_\Phi^{-1})(A\circ\Phi)(J_{\Phi}^{-1})^T$, where $J_\Phi$ is the Jacobian associated with $\Phi$. It is easily seen that $|\det J_\Phi|\equiv1$, which combined with Lemma~\ref{lem:productDMO} implies that $\tilde f,\tilde g, \tilde A\in C_1^{k-1,\tilde\omega}$ for some Dini function $\tilde\omega$. Hence we are concerned with the regularity of the ratio
$$\tilde w=w\circ\Phi=\frac{v\circ\Phi}{u\circ\Phi}.$$
%Let us remark here that in order 
To prove Theorem \ref{thm:BHP} it suffices to show $\tilde w\in C^k$ up to $\Sigma=\{x_n=0\}$. % and this gives $w\in C^k$ up to $\partial\Omega$ composing back with $\Phi^{-1}$.
% {\color{blue}In order to prove Theorem \ref{thm:BHP2} we will show that $\tilde w\in C_{2,\mu_{2}}^{k,\overline\omega}$ where $d\mu_{2}=x_n^2dx$. Then this gives $\tilde w\in C_{2,\tilde\mu}^{k,\overline\omega}$ with $d\tilde\mu=\tilde u^2(x)dx$. \mpab{I THINK IT IS ENOUGH $0<c\leq \tilde u/x_n\leq C$} Hence, composing back with the diffeomorphism $\Phi^{-1}$ and applying Lemma \ref{diffeoDMO}, we will get $w=\tilde w\circ\Phi^{-1}\in C_{2,\mu}^{k,\overline\omega}$ with $d\mu=u^2(x)dx$. TO VERIFY...}

In the proof below we will rename for sake of simplicity of notation $\tilde u,\tilde v,\tilde w,\tilde f,\tilde g, \tilde A, \tilde\omega$ as $u,v,w,f,g,A,\omega$.

\begin{proof}[Proof of Theorem~\ref{thm:BHP}]
After composing with the standard diffeomorphism in \eqref{standard_diffeo}, the ratio solves
\begin{equation*}
    -\div\left(x_n^2\left(\frac{u}{x_n}\right)^2A\nabla w\right)=x_n \left(\frac{u}{x_n}f-\frac{v}{x_n}g\right).
\end{equation*}
Taking $k\in\mathbb N$, we have $A,f,g\in C_1^{k-1,\omega}$. Then by Corollary \ref{cor:HO-u-x_n-reg}, we have that $u/x_n\in C^{k-1}$ and $D^{k-1}_{x'}(u/x_n)\in C_{1,\mu_1}^{0,\tilde \omega}$ (the same for $v$) for some Dini function $\tilde\omega$ and $d\mu_\alpha=x_n^\alpha dx$. Hence, $w$ is a solution to
\begin{equation*}
    -\div\left(x_n^2 \overline A\nabla w\right)=\div(x_n^2 \overline\bff),
\end{equation*}
where $\overline A=\left(\frac{u}{x_n}\right)^2A$ and
$$\overline\bff(x',x_n)=\frac{\vec e_n}{x_n^2}\int_0^{x_n}t \overline f(x',t) \, dt= \vec e_n \int_0^{1}s \overline f(x',sx_n) \, ds,$$
where $\overline f=\frac{u}{x_n}f-\frac{v}{x_n}g$. By Lemma \ref{lem:productDMO}, Lemma \ref{lem:inclusionweight} and Corollary \ref{cor:integ-DMO-k}, we have $\overline A,\overline\bff\in C^{k-1}$ and $D^{k-1}_{x'}\overline A, D^{k-1}_{x'}\overline\bff\in C_{1,\mu_2}^{0,\overline\omega}$, and thus we can conclude by applying Theorem \ref{thm:deg-pde-HO}.
\end{proof}

\subsection{Higher-order boundary Harnack principle across regular zero sets}\label{5}

The aim of this subsection is the proof of Theorem~\ref{thm:BHPRu}. We begin by deriving the $C^{k,2-\mathrm{DMO}}$-regularity of the $A$-harmonic function $u$ for $A\in C_2^{k-1,\omega}$. In fact, in the proposition below, we deal with a more general situation. 

In Theorem~\ref{thm:BHPRu}, we require an $L^2$-DMO type condition on the coefficient $A$, whereas we impose $L^1$-DMO type conditions for the other main results. This is because if $A$ belongs to the $L^1$-DMO type space $C_1^{k-1,\omega}$, we expect only $C^{k}$-regularity of $u$, which is insufficient for our objective. As we will soon observe, Dini mean oscillation of derivatives of $u$ is necessary for the DMO type condition for the new coefficient after flattening the regular nodal set.

\begin{proposition}\label{flatschauder2}
Let $n\geq2$, $k\in\mathbb N$, $\omega$ a Dini function and $u$ be a weak solution to \eqref{flatUEeq}:
\begin{align*}
    \begin{cases}
        -\div(A\D u)=\div\bg&\text{in }B_1^+,\\
        u=0&\text{on }B_1'.
    \end{cases}
    \end{align*}
Let $A,\bg\in C^{k-1,\omega}_{2}(B_1^+)$. Then $u\in C^{k,2-\mathrm{DMO}}_\loc(B_1^+\cup B_1')$. Moreover, if $$
\|A\|_{C^{k-1,\omega}_{2}(B_1^+)}\leq L,
$$ 
then there exists a positive constant $C$ and a Dini modulus of continuity $\overline\omega$ depending only on $n,\lambda,k,L$, and $\omega$ such that
\begin{equation*}
\|u\|_{C_2^{k,\overline\omega}(B^+_{1/2})}\leq C\left(\|u\|_{L^{2}(B_1^+)}+\|\bg\|_{C^{k-1,\omega}_{2}(B_1^+)}\right).
\end{equation*}
\end{proposition}

\begin{proof}
When $k=1$, the result can be inferred from \cite{Don12}*{Theorem~1}. In fact, this theorem concerns the interior $C^{1,2-\mathrm{DMO}}$-regularity of the solution with data of partially Dini mean-oscillation with respect to $x'$-variable. It is worth noting that while the statement of the theorem indicates that the solution is $C^1$, its $C^{1,2-\mathrm{DMO}}$ regularity can be easily inferred from the proof. To apply this result in our context, we take odd-extensions for $u$, $A$ and $\bg$ from $B_1^+$ to $B_1$. Then \eqref{flatUEeq} gives
$$
-\div(A\D u)=\div\bg\quad\text{in }B_1.
$$
Since the extended $A$ and $\bg$ still remain of partially Dini mean-oscillation, the case $k=1$ follows.

We can extend the case $k=1$ to the general case $k\in\mathbb{N}$ by using the induction argument demonstrated in the proof of Proposition~\ref{flatschauder1}.
\end{proof}

By employing the standard flattening of the boundary \eqref{standard_diffeo}, Proposition~\ref{flatschauder2} can be generalized to the boundary Schauder estimate in $C^{k,2-\text{DMO}}$ domains as follows: let $u$ be a solution of \eqref{eqDMOBoundary}, then the following result holds true
\begin{corollary}[Boundary Schauder estimates in $C^{k,2-\mathrm{DMO}}$ domains]\label{CkUnif2}
Let $n\geq2$, $k\in\mathbb N$, $\omega$ a Dini function and $u$ be a weak solution to \eqref{eqDMOBoundary}. Let $A,\bff\in C^{k-1,\omega}_{2}(\Omega\cap B_1)$, $\gamma\in C^{k,\omega}_{2}(B_1')$. Then $u\in C^{k,2-\mathrm{DMO}}_\loc(\overline\Omega\cap B_1)$. Moreover, if $\|A\|_{C^{k-1,\omega}_{2}(\Omega\cap B_1)}+\|\gamma\|_{C^{k,\omega}_{2}(B'_1)}\leq L$, then there exists a positive constant $C$ and a Dini modulus of continuity $\overline\omega$ depending only on $n,\lambda,k,L$, and $\omega$ such that
\begin{equation*}
\|u\|_{C_{2}^{k,\overline\omega}(\Omega\cap B_{1/2})}\leq C\left(\|u\|_{L^{2}(\Omega\cap B_1)}+\|\bff\|_{C^{k-1,\omega}_{2}(\Omega\cap B_1)}\right).
\end{equation*}
\end{corollary}

\begin{remark}\label{cor:uoverx_n2}
By using Proposition \ref{flatschauder2} and Lemma \ref{lem:integ-DMO}, in the $L^2$-DMO setting one can prove the following counterpart of Proposition~\ref{prop:ratio-L1-DMO-k}: 
let us assume the same conditions of Proposition \ref{flatschauder2}. Then there exists a Dini function $\tilde\omega$ %such that $u/x_n\in C_{2}^{k-1,\tilde\omega}(B_{1/2}^+)$. Moreover, if $\|A\|_{C^{k-1,\omega}_{2}(B_1^+)}\leq L$, then there exist a Dini modulus $\tilde\omega$ 
and a positive constant $C$ depending only on $n,\lambda,k,L$, and $\omega$ such that
\begin{equation*}
\left\|\frac{u}{x_n}\right\|_{C_{2}^{k-1,\tilde\omega}(B^+_{1/2})}\leq C\left(\|u\|_{L^{2}(B_1^+)}+\|\bg\|_{C^{k-1,\omega}_{2}(B_1^+)}\right).
\end{equation*}
\end{remark}

Now, let us return to Theorem~\ref{thm:BHPRu}. To prove it, we use the following strategy: first we localize the problem around a given point on the regular part $R(u)$ of the nodal set $Z(u)=u^{-1}\{0\}$, where $u$ is a given $A$-harmonic function, i.e., local solution to $\div(A\nabla u)=0$. For simplicity $0\in R(u)\cap B_1$, $u$ is $A$-harmonic in $B_1$ and $S(u)\cap B_1=\emptyset$. Considering another $A$-harmonic function $v$ in $B_1$ such that locally $Z(u)\subseteq Z(v)$, then the ratio $w=v/u$ is solution to
$$\div(u^2A\nabla w)=0\qquad\mathrm{in \ }B_1.$$
We would like to straighten the regular nodal set and get the regularity estimates for the ratio from both sides of $R(u)$, and finally glue them together across the free interface. Under $L^2$-DMO type assumptions on coefficients, specifically $A\in C_{2}^{k-1,\omega}$, the solution $u$ belongs to $C_{2}^{k,\tilde\omega}$-spaces, by the interior counterpart of Proposition \ref{flatschauder2}. We would like to stress the fact that the implicit function theorem cannot ensure that $R(u)\cap B_1$ is a hypersurface of class $C_{2}^{k,\tilde\omega}$ since DMO type conditions are not preserved under restrictions to lower dimensional subsets. For this reason we need to make use of a hodograph transformation \cites{Fri34,KinNir77} which is, in the present case, a diffeomorphism of class $C_{2}^{k,\tilde\omega}$ which flattens the level sets of $u$.

We consider the $A$-harmonic function $u$ in $B_1$ with $0\in R(u)$, $S(u)\cap B_1=\emptyset$ and $|\D u|\ge c>0$. This nondegeneracy condition allows us assume, %to select a partial derivative which is not degenerate, i.e., 
up to rotations, $|\partial_n u|\geq c>0$. Then, we can define the following diffeomorphism
\begin{equation}\label{udiffeo}
    \Psi(x',x_n)=(x',u(x',x_n)).
\end{equation}
Denoting by $x=(x',x_n)$ the original coordinates and $y=(y',y_n)$ the new coordinates,
\begin{equation*}
    \begin{cases}
        y'=x' \\
        y_n=u(x',x_n).
    \end{cases}
\end{equation*}
Up to dilations, $\Psi$ maps $\{u>0\}\cap B_1$ into $B_1^+=B_1\cap\{y_n>0\}$, $\{u<0\}\cap B_1$ into $B_1^-=B_1\cap\{y_n<0\}$ and $R(u)\cap B_1$ into $B_1'=B_1\cap\{y_n=0\}$. In particular, for any $t\in\R$, the level set $u^{-1}\{t\}\cap B_1$ is locally mapped into the hyperplane $\{y_n=t\}\cap B_1$. Then,
\begin{equation*}
    u\circ\Psi^{-1}(y',y_n)=y_n.
\end{equation*}
The Jacobian associated with $\Psi$ is given by
\begin{align*}
J_\Psi(x)=\left(\begin{array}{c|c}
\mathbb{I}_{n-1}&{\mathbf 0}\\\hline
(\nabla_{x'}u(x))^T&\partial_nu(x)
\end{array}\right), \qquad \mathrm{with}\quad |\mathrm{det} \, J_\Psi(x)|=|\partial_nu(x)|\geq c>0,
\end{align*}
and hence $\Psi$ is locally invertible and bi-Lipschitz by the implicit function theorem. In fact
$$J_{\Psi^{-1}}=J_{\Psi}^{-1}\circ\Psi^{-1};$$
that is,
\begin{align*}
J_{\Psi^{-1}}(y)=\left(\begin{array}{c|c}
\mathbb{I}_{n-1}&{\mathbf 0}\\\hline
-(\nabla_{x'}u\circ\Psi^{-1}(y))^T/\partial_n u\circ\Psi^{-1}(y)&1/\partial_n u\circ\Psi^{-1}(y)
\end{array}\right),
\end{align*}
with
$$|\mathrm{det} \, J_{\Psi^{-1}}(y)|=\frac{1}{|\partial_nu\circ\Psi^{-1}(y)|}\geq\frac{1}{\|\nabla u\|_{\infty}}>0.$$

Then, up to dilations, $\tilde w=w\circ\Psi^{-1}$ solves
\begin{equation*}
\div\left(x_n^2\tilde A\nabla \tilde w\right)=0 \qquad\mathrm{in \ } B_1,
\end{equation*}
where
\begin{equation*}
\tilde A=(J_{\Psi^{-1}}^{-1}) (A \circ\Psi^{-1}) (J_{\Psi^{-1}}^{-1})^T|\mathrm{det} \, J_{\Psi^{-1}}|=\frac{J_\Psi A J_\Psi^T}{|\mathrm{det} \, J_{\Psi}|}\circ\Psi^{-1}.
\end{equation*}
Notice that DMO type conditions are preserved under composition with this diffeomorphism%, since the bounds $c\leq|\mathrm{det} \, J_\Psi(x)|\leq C$ give a pushforward measure which is equivalent to the Lebesgue measure
. Then, when $u\in C^{k,\tilde\omega}_2$ the new matrix $\tilde A$ belongs to $C_{2}^{k-1,\tilde\omega}\subset C_{1}^{k-1,\tilde\omega}\subset C_{1,\mu_2}^{k-1,\tilde\omega}$ with $d\mu_2=x_n^2dx$. Now we shall prove the $C^k$ regularity of $\tilde w$ across $\Sigma=\{y_n=0\}$ and composing back with $\Psi$, which is of class $C_{2}^{k,\tilde\omega}$, this will give the same regularity for $w=\tilde w\circ\Psi$. Let us denote $\tilde A,\tilde w, \tilde\omega$ by $A,w,\omega$.

\begin{proof}[Proof of Theorem \ref{thm:BHPRu}]
After applying the diffeomorphism in \eqref{udiffeo}, the ratio solves
\begin{equation*}
\div\left(x_n^2 A\nabla  w\right)=0 \qquad\mathrm{in \ } B_1,
\end{equation*}
where $A$ belongs to $C_{1,\mu_2}^{k-1,\omega}$ for some Dini function $\omega$, and solves the same problem separately on the upper and lower half balls $B_1^+, B_1^-$ with conormal boundary condition at $B_1'$. Hence, by applying Theorem \ref{thm:deg-pde-HO} on the two half balls separately, we get that $w$ belongs to $C^{k,\sigma}(B_{1/2}^+)$ and $C^{k,\sigma}(B_{1/2}^-)$ for some modulus $\sigma$ which is $\be$-nonincreasing (i.e., $r\mapsto\sigma_\be(r)r^{-\be}$ is nonincreasing) for some $\beta\in(0,1]$.
% Given $0<\be<1$, $\sigma_\be(r):=r^\be\sup_{s\in [r,1]}\frac{\sigma(s)}{s^\be}$ satisfies that $\sigma\le \sigma_\be$ and $\sigma_\be$ is $\be$-nonincreasing (i.e., $r\mapsto\sigma_\be(r)r^{-\be}$ is nonincreasing).
Finally, we can apply the gluing lemma in \cite{TerTorVit22}*{Lemma 2.11} which can be generalized to the case of the present modulus $\sigma$ due to the validity of \cite{JeoVit23}*{Lemma 3.9 and Lemma 3.10}. Let us stress that the validity of the gluing lemma relies on the boundary condition \eqref{Neumann}.
\end{proof}

%%%%%%%%%%%%%%%%%%%%%%%%%%%%%%%%%%%%%%%%%%%%%%%%%%%%%%%%%%%%%%%%%%%%%%%%%%%%%

\appendix
\section{Properties of DMO functions}\label{A}

The following doubling property of the measure can be directly checked.
\begin{lemma}
    \label{lem:doubl}
For $\al>-1$, let $d\mu=x_n^\al dx$. Suppose that $\cD$ is a Lipschitz and convex domain in $B_1^+$. Then there exists a constant $C>0$, depending only on $n$, $\al$, and the Lipschitz constant of $\cD$, such that
\begin{align*}
    \mu(B_{2r}(x_0)\cap \cD)\le C\mu(B_r(x_0)\cap \cD)
\end{align*}
for any $x_0\in \cD$ and $0<r<\diam \cD$.
\end{lemma}

%\begin{proof}
%Given $x_0=(x_0',(x_0)_n)\in \cD$ and $0<r<\diam \cD$, let $\bar x_0:=(x_0',0)\in B'_1$ and $d:=(x_0)_n>0$. There are three possibilities
%$$
%r\le d/4,\quad d/4<r<2d,\quad r\ge 2d.
%$$
%Before we examine these three cases, we note that there is a constant $c_0>0$, depending only on the Lipschitz constant of $\cD$, such that $B_r(x_0)\cap \cD\supset B_{c_0r}(y_0)$ for some point $y_0$. In particular,
%\begin{align}\label{eq:doubl-meas}
%|B_r(x_0)\cap \cD|\ge c_1r^n,\quad c_1=c_1(n,\cD)>0.
%\end{align}
%
%\medskip\noindent\emph{Case 1.} We first consider the case $r\le d/4$. Since $d/2<x_n<3d/2$ for every $x=(x',x_n)\in B_{2r}(x_0)$, we readily have by \eqref{eq:doubl-meas} that
%$$
%\mu(B_r(x_0)\cap \cD)\ge cr^nd^\al\ge c\mu(B_{2r}(x_0)\cap \cD).
%$$
%
%\medskip\noindent\emph{Case 2.} Suppose $d/4<r<2d$. We then simply have
%$$
%\mu(B_{2r}(x_0)\cap \cD)\le \mu(B_{2r+d}^+(\bar x_0))\le Cr^{n+\al}.
%$$
%On the other hand, by using \eqref{eq:doubl-meas} and the fact that $x_n\in (r/4,5r)$ whenever $x\in B_{r/4}(x_0)$, we get
%$$
%\mu(B_r(x_0)\cap \cD)\ge \mu(B_{r/4}(x_0)\cap \cD)\ge cr^{n+\al}.
%$$
%
%\medskip\noindent\emph{Case 3.} Suppose $r\ge 2d$. We argue as in Case 2 to get
%$$
%\mu(B_{2r}(x_0)\cap \cD)\le \mu(B_{2r+d}^+(\bar x_0))\le Cr^{n+\al}.
%$$
%In addition, we observe that
%$$
%B_r(x_0)\cap \cD\supset B_{c_0r}(y_0)\supset B_{c_0r/4}(y_0+c_0r/2\vec{e}_n).
%$$
%Since $c_0r/4<x_n<2r$ for any $x\in B_{c_0r/4}(y_0+c_0r/2\vec{e}_n)$, we have
%$$
%\mu(B_r(x_0)\cap \cD)\ge \mu(B_{c_0r/4}(y_0+c_0r/2\vec{e}_n))\ge cr^{n+\al}.
%$$
%This completes the proof.
%\end{proof}

\begin{lemma}\label{lem:productDMO} Let $f$ and $g$ be bounded functions in a domain $\Omega$. Let $\mu$ be a Radon measure and $q\in[1,+\infty)$. If $f$ and $g$ are of $L^q(d\mu)$-DMO in $\Omega$ then so is $fg$.
\end{lemma}

\begin{proof}
Let $x_0\in \Omega$ and $0<r<1$ be given. Then, for $\Omega_r(x_0)=\Omega\cap B_r(x_0)$, we have
\begin{align*}
    &\dashint_{\Omega_r(x_0)}|f(x)g(x)-\mean{fg}^\mu_{\Omega_r(x_0)}|^q\,d\mu(x)\\
    &\le C\dashint_{\Omega_r(x_0)}|f(x)g(x)-f(x)\mean{g}^\mu_{\Omega_r(x_0)}|^q\,d\mu(x)\\
    &\qquad+C\dashint_{\Omega_r(x_0)}|f(x)\mean{g}_{\Omega_r(x_0)}^\mu-\mean{fg}_{\Omega_r(x_0)}^\mu|^q\,d\mu(x)\\
    &\le C\dashint_{\Omega_r(x_0)}|f(x)|^q|g(x)-\mean{g}_{\Omega_r(x_0)}^\mu|^q\,d\mu(x)\\
    &\qquad +C\dashint_{\Omega_r(x_0)}\dashint_{\Omega_r(x_0)}|f(x)-f(y)|^q \, |g(y)|^q\,d\mu(y)\,d\mu(x)\\
    &\le C\|f\|^q_\infty[\eta^{q,\mu}_g(r)]^q+C\|g\|^q_\infty[\eta_f^{q,\mu}(r)]^q.\qedhere
\end{align*}
\end{proof}

\begin{lemma}\label{lem:inclusionweight}
Let $\al\ge\be>-1$ and $q\in [1,+\infty)$. If $f$ is of $L^q(x_n^\be dx)$-DMO in $B_1^+$, then it is of $L^q(x_n^\al dx)$-DMO in $B_{1/2}^+$.
\end{lemma}

\begin{proof}
Let $x_0=(x_0',(x_0)_n)\in B_{1/2}^+$ and $0<r<1/6$ be given. For simplicity, we write $d\mu_\al(x)=x_n^\al dx$, $\Omega_r(x_0):=B_r(x_0)\cap B_{1/2}^+$, $d:=(x_0)_n$, and $\bar x_0:=(x_0',0)$. We consider two cases:
$$
\text{either $d/2\le r$ or $d/2>r$.}
$$
\noindent\emph{Case 1.} If $d/2\le r$, then we use $B_r(x_0)\subset B_{3r}(\bar x_0)$ and apply Lemma~\ref{lem:doubl} to get
\begin{align*}
    \dashint_{\Omega_r(x_0)}|f-\mean{f}_{\Omega_r(x_0)}^{\mu_\al}|^qd\mu_\al
    &\le C\dashint_{B^+_{3r}(\bar x_0)}\dashint_{B_{3r}^+(\bar x_0)}|f(x)-f(y)|^qd\mu_\al(x)d\mu_\al(y)\\
    &\le \frac{C}{r^{2(n+\al)}}\int_{B_{3r}^+(\bar x_0)}\int_{B_{3r}^+(\bar x_0)}|f(x)-f(y)|^qx_n^\al y_n^\al dxdy\\
    &\le \frac{C}{r^{2(n+\be)}}\int_{B_{3r}^+(\bar x_0)}\int_{B_{3r}^+(\bar x_0)}|f(x)-f(y)|^qx_n^\be y_n^\be dxdy\\
    &\le C\dashint_{B_{3r}^+(\bar x_0)}|f-\mean{f}^{\mu_\be}_{B_{3r}^+(\bar x_0)}|^qd\mu_\be\,\le\, C[\eta_f^{q,\mu_\be}(3r)]^q.
\end{align*}

\noindent\emph{Case 2.} If $d/2>r$, then by using that $d/2\le x_n<3d/2$ for every $x=(x',x_n)\in B_r(x_0)$, we can easily obtain
\begin{equation*}
    \dashint_{\Omega_r(x_0)}|f-\mean{f}_{\Omega_r(x_0)}^{\mu_\al}|^qd\mu_\al\le C\dashint_{\Omega_r(x_0)}|f-\mean{f}_{\Omega_r(x_0)}^{\mu_\be}|^qd\mu_\be\le C[\eta_f^{q,\mu_\be}(r)]^q.\qedhere
\end{equation*}
\end{proof}

\begin{lemma}\label{lem:integ-DMO}
    For $\al>-1$, let $d\mu=x_n^\al dx$. Let $q\in[1,+\infty)$ and $\beta>0$.
    If $f$ is of $L^{q}(d\mu)$-DMO in $B_1^+$, then $\hat f(x):=\int_0^1s^\beta f(x',sx_n)ds$ is of $L^{q}(d\mu)$-DMO in $B^+_{1}$.
\end{lemma}

\begin{proof}
Given $x_0\in B_{1}^+$ and $0<r<\frac1{2\sqrt n}$, we have by applying Minkowski's integral inequality and Jensen's inequality
\begin{align*}
    &\left(\dashint_{B_r(x_0)\cap B_1^+}|\hat f-\mean{\hat f}^\mu_{B_r(x_0)\cap B_1^+}|^{q} d\mu\right)^{1/{q}}\\
    & =\left(\dashint_{B_r(x_0)\cap B_1^+}\left|\int_0^1s^\be(f(x',sx_n)-\mean{f(\cdot,s\cdot)}^\mu_{B_r(x_0)\cap B_1^+})ds\right|^{q}d\mu(x)\right)^{1/{q}}\\
    &\le \int_0^1s^\be\left(\dashint_{B_r(x_0)\cap B_1^+}|f(x',sx_n)-\mean{f(\cdot,s\cdot)}^\mu_{B_r(x_0)\cap B_1^+}|^{q}d\mu(x)\right)^{1/{q}}ds\\
    &\le \int_0^1s^{\beta}\left(\dashint_{B_r(x_0)\cap B_1^+}\dashint_{B_r(x_0)\cap B_1^+}|f(x',sx_n)-f(y',sy_n)|^{q} d\mu(x)d\mu(y)\right)^{1/{q}}ds.
\end{align*}
We claim that there is a dimensional constant $C>0$ such that for each $0<s<1$,
\begin{align}
    \label{eq:integ-DMO-claim}
    &\left(\dashint_{B_r(x_0)\cap B_1^+}\dashint_{B_r(x_0)\cap B_1^+}|f(x',sx_n)-f(y',sy_n)|^{q} d\mu(x)d\mu(y)\right)^{1/{q}}\notag\\
    &\le \frac{C\eta_f^{q,\mu}(2\sqrt n sr)}{s}.
\end{align}
Suppose now the claim is true. Then, we readily have
$$
\left(\dashint_{B_r(x_0)\cap B_1^+}|\hat f-\mean{\hat f}^\mu_{B_r(x_0)\cap B_1^+}|^{q} d\mu\right)^{1/{q}}\le C\int_0^1 s^{\beta-1}\eta_f^{q,\mu}(2\sqrt n sr)ds.
$$
Here, $\omega(r):=\int_0^1 s^{\beta-1}\eta_f^{q,\mu}(2\sqrt n sr)ds$ is a Dini function since
\begin{align*}
    \int_0^{\frac{1}{2\sqrt n}}\frac{\omega(r)}rdr&=\int_0^1s^{\be-1}\int_0^{\frac1{2\sqrt n}}\frac{\eta_f^{q,\mu}(2\sqrt n sr)}rdrds=\int_0^1s^{\be-1}\int_0^s\frac{\eta_f^{q,\mu}(t)}tdtds\\
    &\le \left(\int_0^1\frac{\eta_f^{q,\mu}(t)}tdt\right)\left(\int_0^1s^{\be-1}ds\right)<+\infty.
\end{align*}
To close the argument, we need to verify \eqref{eq:integ-DMO-claim}, which by a change of variables is equivalent to
\begin{equation}    \label{eq:integ-DMO}
\begin{aligned}
    &\int_{E_r^s(x_0)\cap B_1^+}\int_{E_r^s(x_0)\cap B_1^+}|f(x)-f(y)|^{q} d\mu(x)d\mu(y)\\
    &\le C[\eta_f^{q,\mu}(2\sqrt n sr)]^{q}[\mu(B_r(x_0)\cap B_1^+)]^2s^{2-q+2\al},
\end{aligned}
\end{equation}
where $E_r^s(x_0)$ are ellipsoids defined by
$$
E_r^s(x_0):=\left\{x=(x',x_n)\in \R^n\,:\, |x'-x_0'|^2+\frac{|x_n-s(x_0)_n)|^2}{s^2}<r^2\right\}.
$$
To prove \eqref{eq:integ-DMO}, we cover the ellipsoid $E_r^s(x_0)$ by hypercubes of length $2sr$ in the following way:
\begin{align*}
        &\text{- each center of the hypercube lies on $\{x_n=s(x_0)_n\}$,}\\
        &\text{- every intersection of two hypercubes has zero $n$-dimensional measure.}
\end{align*}
Note that the number of hypercubes is bounded by $C/{s^{n-1}}$ for some dimensional constant $C>0$. Next, we cover each hypercube by the concentric ball of radius $\sqrt n sr$. Then we can write
$$
E_r^s(x_0)\subset \bigcup_{B\in \mathcal{F}}B,
$$
where each element $B$ is a ball of radius $\sqrt n sr$ centered on $\{x_n=s(x_0)_n\}$ and $n(\mathcal{F})\le C/{s^{n-1}}$.

We assert that given arbitrary two balls in $\mathcal{F}$, say $\bar B$ and $\tilde B$, it holds that
\begin{equation}
    \label{eq:integ-DMO-est}
\begin{aligned}
&    \int_{\bar B\cap B_1^+}\int_{\tilde B\cap B_1^+}|f(x)-f(y)|^{q}d\mu(x)d\mu(y)\\
 &   \le Cs^{-q}[\mu(B_{2\sqrt n sr}(sx_0)\cap B_1^+)]^2[\eta_f^{q,\mu}(2\sqrt n sr)]^{q}.
\end{aligned}
\end{equation}
Indeed, if $\bar B$ and $\tilde B$ are adjacent (i.e., $\bar B\cap \tilde B\neq\emptyset$), then $\bar B\cup \tilde B\subset \hat B$ for a ball $\hat B$ centered on $\{x_n=s(x_0)_n\}$ of radius $2\sqrt n sr$. This readily gives
\begin{align*}
    &\int_{\bar B\cap B_1^+}\int_{\tilde B\cap B_1^+}|f(x)-f(y)|^{q}d\mu(x)d\mu(y)\\
    &\qquad\le \int_{\hat B\cap B_1^+}\int_{\hat B\cap B_1^+}|f(x)-f(y)|^{q}d\mu(x)d\mu(y)\\
    &\qquad\le C[\mu(B_{2\sqrt n sr}(sx_0)\cap B_1^+)]^2[\eta_f^{q,\mu}(2\sqrt n sr)]^{q}.
\end{align*}
On the other hand, if $\bar B$ and $\tilde B$ are not adjacent, we connect them by a chain of balls in $\mathcal{F}$, with a length of the chain bounded by $C/s$ for some $C=C(n)>0$. That is, we consider a sequence of balls $\bar B=B^0, B^1,\cdots,B^{J-1},B^J=\tilde B$ such that $J\le C/s$, $B^{j-1}\cap B^j\neq\emptyset$ and $B^j\in \mathcal{F}$, $1\le j\le J$. Then by H\"older's inequality
\begin{align*}
    &\int_{B^0\cap B_1^+}\int_{B^J\cap B_1^+}|f(x^0)-f(x^J)|^{q} d\mu(x^J)d\mu(x^0)\\
    &\qquad\begin{multlined}\le\int_{B^0\cap B_1^+}\int_{B^J\cap B_1^+}\dashint_{B^1\cap B_1^+}\cdots\dashint_{B^{J-1}\cap B_1^+}\\
    \qquad\qquad J^{q-1}\sum_{j=1}^J|f(x^{j-1})-f(x^j)|^{q} d\mu(x^{J-1})\cdots d\mu(x^1)d\mu(x^J)d\mu(x^0)\end{multlined}\\
    &\qquad= J^{q-1}\sum_{j=1}^J\int_{B^{j-1}\cap B_1^+}\int_{B^j\cap B_1^+}|f(x^{j-1})-f(x^j)|^{q} d\mu(x^j)d\mu(x^{j-1})\\
    &\qquad\le CJ^{q}[\mu(B_{2\sqrt n sr}(sx_0)\cap B_1^+)]^2[\eta_f^{q,\mu}(2\sqrt n sr)]^{q}\\
    &\qquad\le Cs^{-q}[\mu(B_{2\sqrt n sr}(sx_0)\cap B_1^+)]^2[\eta_f^{q,\mu}(2\sqrt n sr)]^{q},
\end{align*}
and hence \eqref{eq:integ-DMO-est} is proved.

Now, by using \eqref{eq:integ-DMO-est}, we obtain
\begin{align*}
    &\int_{E_r^s(x_0)\cap B_1^+}\int_{E_r^s(x_0)\cap B_1^+}|f(x)-f(y)|^{q} d\mu(x)d\mu(y)\\
    &\qquad \le \sum_{1\le i,j\le C/{s^{n-1}}}\int_{B_{sr}(z_i)}\int_{B_{sr}(z_j)}|f(x)-f(y)|^{q} d\mu(x)d\mu(y)\\
    &\qquad\le Cs^{-(2n-2+q)}[\mu(B_{2\sqrt n sr}(sx_0)\cap B_1^+)]^2[\eta_f^{q,\mu}(2\sqrt n sr)]^{q}.
\end{align*}
Finally, by using $\mu(B_{2\sqrt n sr}(sx_0)\cap B_1^+)\le C\mu(B_{sr}(sx_0)\cap B_1^+)\le Cs^{n+\al}\mu(B_r(x_0)\cap B_1^+)$, we conclude \eqref{eq:integ-DMO}.
\end{proof}

Lemma~\ref{lem:integ-DMO} has an immediate corollary.
\begin{corollary}\label{cor:integ-DMO-k}
   Let $\al$, $\mu$, $q$ and $\be$ be as in Lemma~\ref{lem:integ-DMO}. Given $k\in \mathbb{N}$ and a Dini function $\omega$, if $f\in C^{k,\omega}_{q,\mu}(B_1^+)$, then $\hat f(x):=\int_0^1s^\be f(x',sx_n)ds\in C^{k,\bar\omega}_{q,\mu}(B^+_{1})$ for some Dini function $\bar\omega$.
\end{corollary}

{\bf Data Availability Statements.} Data sharing not applicable to this article as
no datasets were generated or analysed during the current study.

\section*{Acknowledgment}
H. Dong was partially supported by a Simons fellowship, grant no. 007638 and the NSF under agreement DMS-2055244. S. Jeon is supported by the Academy of Finland grant 347550. S. Vita is research fellow of Istituto Nazionale di Alta Matematica INDAM group GNAMPA, supported by the MUR funding for Young Researchers - Seal of Excellence SOE\_0000194 \emph{(ADE) Anomalous diffusion equations: regularity and geometric properties of solutions and free boundaries}, and supported also by the PRIN project 2022R537CS \emph{$NO^3$ - Nodal Optimization, NOnlinear elliptic equations, NOnlocal geometric problems, with a focus on regularity}.

%\section*{Declaration}

%Funding and/or Conflicts of interests: The authors declare that there are no financial or non-financial conflict of interests.

%%%%%%%%%%%%%%%%%%%%%%%%%%%%%%%%%%%%%%


\begin{bibdiv}
\begin{biblist}


\bib{ApuNaz16}{article}{
   author={Apushkinskaya, Darya E.},
   author={Nazarov, Alexander I.},
   title={A counterexample to the Hopf-Oleinik lemma (elliptic case)},
   journal={Anal. PDE},
   volume={9},
   date={2016},
   number={2},
   pages={439--458},
   issn={2157-5045},
   review={\MR{3513140}},
   doi={10.2140/apde.2016.9.439},
   }


\bib{ApuNaz19}{article}{
   author={Apushkinskaya, Darya E.},
   author={Nazarov, Alexander I.},
   title={On the boundary point principle for divergence-type equations},
   journal={Atti Accad. Naz. Lincei Rend. Lincei Mat. Appl.},
   volume={30},
   date={2019},
   number={4},
   pages={677--699},
   issn={1120-6330},
   review={\MR{4030346}},
   doi={10.4171/RLM/867},
   }



\bib{BanGar16}{article}{
   author={Banerjee, Agnid},
   author={Garofalo, Nicola},
   title={A parabolic analogue of the higher-order comparison theorem of De
   Silva and Savin},
   journal={J. Differential Equations},
   volume={260},
   date={2016},
   number={2},
   pages={1801--1829},
   issn={0022-0396},
   review={\MR{3419746}},
   doi={10.1016/j.jde.2015.09.044},
}


\bib{ChoKimLee20}{article}{
   author={Choi, Jongkeun},
   author={Kim, Seick},
   author={Lee, Kyungrok},
   title={Gradient estimates for elliptic equations in divergence form with partial {D}ini mean oscillation coefficients},
   journal={J. Korean Math. Soc.},
   volume={57},
   date={2020},
   number={6},
   pages={1509--1533},
   issn={0304-9914},
   review={\MR{4169354}},
   doi={10.4134/JKMS.j190777},
   }


\bib{DeSSav15}{article}{
   author={De Silva, Daniela},
   author={Savin, Ovidiu},
   title={A note on higher regularity boundary {H}arnack inequality},
   journal={Discrete Contin. Dyn. Syst.},
   volume={35},
   date={2015},
   number={12},
   pages={6155--6163},
   issn={1078-0947},
   review={\MR{3393271}},
   doi={10.3934/dcds.2015.35.6155},
   }


\bib{Don12}{article}{
   author={Dong, Hongjie},
   title={Gradient estimates for parabolic and elliptic systems from linear laminates},
   journal={Arch. Ration. Mech. Anal.},
   volume={205},
   date={2012},
   number={1},
   pages={119--149},
   issn={0003-9527},
   review={\MR{2927619}},
   doi={10.1007/s00205-012-0501-z},
}


\bib{DonEscKim18}{article}{
   author={Dong, Hongjie},
   author={Escauriaza, Luis},
   author={Kim, Seick},
   title={On {$C^1$}, {$C^2$}, and weak type-{$(1,1)$} estimates for linear elliptic operators: part {II}},
   journal={Math. Ann.},
   volume={370},
   date={2018},
   number={1-2},
   pages={447--489},
   issn={0025-5831},
   review={\MR{3747493}},
   doi={10.1007/s00208-017-1603-6},
}


\bib{DonKim17}{article}{
   author={Dong, Hongjie},
   author={Kim, Seick},
   title={On {$C^1$}, {$C^2$}, and weak type-{$(1,1)$} estimates for linear elliptic operators},
   journal={Comm. Partial Differential Equations},
   volume={42},
   date={2017},
   number={3},
   pages={417--435},
   issn={0360-5302},
   review={\MR{3620893}},
   doi={10.1080/03605302.2017.1278773},
}


\bib{DonLeeKim20}{article}{
   author={Dong, Hongjie},
   author={Lee, Jihoon},
   author={Kim, Seick},
   title={On conormal and oblique derivative problem for elliptic equations with {D}ini mean oscillation coefficients},
   journal={Indiana Univ. Math. J.},
   volume={69},
   date={2020},
   number={6},
   pages={1815--1853},
   issn={0022-2518},
   review={\MR{4170081}},
   doi={10.1512/iumj.2020.69.8028},
}


\bib{DonXu19}{article}{
   author={Dong, Hongjie},
   author={Xu, Longjuan},
   title={Gradient estimates for divergence form elliptic systems arising from composite material},
   journal={SIAM J. Math. Anal.},
   volume={51},
   date={2019},
   number={3},
   pages={2444--2478},
   issn={0036-1410},
   review={\MR{3961984}},
   doi={10.1137/18M1226658},
}


\bib{DP_CVPDE21}{article}{
   author={Dong, Hongjie},
   author={Phan, Tuoc},
   title={Regularity for parabolic equations with singular or degenerate
   coefficients},
   journal={Calc. Var. Partial Differential Equations},
   volume={60},
   date={2021},
   number={1},
   pages={Paper No. 44, 39},
   issn={0944-2669},
   review={\MR{4204570}},
}

\bib{DP_TAMS21}{article}{
   author={Dong, Hongjie},
   author={Phan, Tuoc},
   title={Parabolic and elliptic equations with singular or degenerate
   coefficients: the Dirichlet problem},
   journal={Trans. Amer. Math. Soc.},
   volume={374},
   date={2021},
   number={9},
   pages={6611--6647},
   issn={0002-9947},
   review={\MR{4302171}},
}


\bib{DonPha20}{article}{
   author={Dong, Hongjie},
   author={Phan, Tuoc},
   title={On parabolic and elliptic equations with singular or degenerate coefficients},
   journal={Indiana Univ. Math. J.},
   volume={72},
   date={2023},
   number={4},
   pages={1461--1502},
   issn={0022-2518},
   review={\MR{4637368}},
}


\bib{DP_JFA}{article}{
   author={Dong, Hongjie},
   author={Phan, Tuoc},
   title={Weighted mixed-norm $L_p$ estimates for equations in
   non-divergence form with singular coefficients: the Dirichlet problem},
   journal={J. Funct. Anal.},
   volume={285},
   date={2023},
   number={2},
   pages={Paper No. 109964, 43},
   issn={0022-1236},
   review={\MR{4571874}},
}


 \bib{Fri34}{article}{
   author={Friedrichs, Kurt},
   title={\"{U}ber ein Minimumproblem f\"{u}r Potentialstr\"{o}mungen mit
   freiem Rande},
   language={German},
   journal={Math. Ann.},
   volume={109},
   date={1934},
   number={1},
   pages={60--82},
   issn={0025-5831},
   review={\MR{1512880}},
   doi={10.1007/BF01449125},
}



\bib{JeoVit23}{article}{
   author={Jeon, Seongmin},
   author={Vita, Stefano},
   title={Higher order boundary Harnack principles in Dini type domains},
   journal={J. Differential Equations},
   volume={412},
   date={2024},
   pages={808-856},
   doi={10.1016/j.jde.2024.08.059},
}



\bib{KinNir77}{article}{
   author={Kinderlehrer, David},
   author={Nirenberg, Louis},
   title={Regularity in free boundary problems},
   journal={Ann. Scuola Norm. Sup. Pisa Cl. Sci.},
   volume={4},
   date={1977},
   number={2},
   pages={373--391},
   issn={0391-173X},
   review={\MR{0440187}},
}



\bib{Kuk22}{article}{
   author={Kukuljan, Teo},
   title={Higher order parabolic boundary Harnack inequality in $C^1$ and
   $C^{k,\alpha}$ domains},
   journal={Discrete Contin. Dyn. Syst.},
   volume={42},
   date={2022},
   number={6},
   pages={2667--2698},
   issn={1078-0947},
   review={\MR{4421508}},
   doi={10.3934/dcds.2021207},
}


\bib{Li17}{article}{
   author={Li, Yanyan},
   title={On the {$C^1$} regularity of solutions to divergence form elliptic systems with {D}ini-continuous coefficients},
   journal={Chinese Ann. Math. Ser. B},
   volume={38},
   date={2017},
   pages={489--496},
   issn={0252-9599},
   review={\MR{3615500}},
   doi={10.1007/s11401-017-1079-4},
}



 \bib{LinLin22}{article}{
   author={Lin, Fanghua},
   author={Lin, Zhengjiang},
   title={Boundary Harnack principle on nodal domains},
   journal={Sci. China Math.},
   volume={65},
   date={2022},
   number={12},
   pages={2441--2458},
   issn={1674-7283},
   review={\MR{4514975}},
   doi={10.1007/s11425-022-2016-3},
}


\bib{LogMal15}{article}{
   author={Logunov, Alexander},
   author={Malinnikova, Eugenia},
   title={On ratios of harmonic functions},
   journal={Adv. Math.},
   volume={274},
   date={2015},
   pages={241--262},
   issn={0001-8708},
   review={\MR{3318150}},
   doi={10.1016/j.aim.2015.01.009},
}


\bib{LogMal16}{article}{
   author={Logunov, Alexander},
   author={Malinnikova, Eugenia},
   title={Ratios of harmonic functions with the same zero set},
   journal={Geom. Funct. Anal.},
   volume={26},
   date={2016},
   number={3},
   pages={909--925},
   issn={1016-443X},
   review={\MR{3540456}},
   doi={10.1007/s00039-016-0369-4},
}



\bib{RenSirSoa22}{article}{
   author={Rend\'{o}n, Fiorella},
   author={Sirakov, Boyan},
   author={Soares, Mayra},
   title={Boundary weak Harnack estimates and regularity for elliptic PDE in
   divergence form},
   journal={Nonlinear Anal.},
   volume={235},
   date={2023},
   pages={Paper No. 113331, 13},
   issn={0362-546X},
   review={\MR{4617058}},
   doi={10.1016/j.na.2023.113331},
}


\bib{SirTerVit21a}{article}{
   author={Sire, Yannick},
   author={Terracini, Susanna},
   author={Vita, Stefano},
   title={Liouville type theorems and regularity of solutions to degenerate
   or singular problems part I: even solutions},
   journal={Comm. Partial Differential Equations},
   volume={46},
   date={2021},
   number={2},
   pages={310--361},
   issn={0360-5302},
   review={\MR{4207950}},
   doi={10.1080/03605302.2020.1840586},
}
 
 
 

 \bib{SirTerVit21b}{article}{
   author={Sire, Yannick},
   author={Terracini, Susanna},
   author={Vita, Stefano},
   title={Liouville type theorems and regularity of solutions to degenerate
   or singular problems part II: odd solutions},
   journal={Math. Eng.},
   volume={3},
   date={2021},
   number={1},
   pages={Paper No. 5, 50},
   review={\MR{4144100}},
   doi={10.3934/mine.2021005},
}



 \bib{TerTorVit22}{article}{
   author={Terracini, Susanna},
   author={Tortone, Giorgio},
   author={Vita, Stefano},
   title={Higher order boundary Harnack principle via degenerate equations},
   journal={Arch. Ration. Mech. Anal.},
   volume={248},
   date={2024},
   number={2},
   pages={Paper No. 29, 44},
   issn={0003-9527},
   review={\MR{4726059}},
   doi={10.1007/s00205-024-01973-1},
}


\end{biblist}
\end{bibdiv}
\end{document}